\pdfoutput=1
\documentclass[11pt, a4paper, reqno]{amsart}

\usepackage[margin=2.5cm,top=4cm,bottom=3cm,footskip=1cm,headsep=2cm]{geometry}

\usepackage[english]{babel}
\usepackage[utf8]{inputenc}
\usepackage{amsmath, amssymb, amsthm, amsfonts}
\usepackage{graphicx}
\usepackage{hyperref}
\usepackage{amsaddr}
\usepackage{caption}
\usepackage{subcaption}
\usepackage{tikz}
\usetikzlibrary{positioning}

\theoremstyle{plain}
\newtheorem{thm}{Theorem}[section]
\newtheorem{lem}[thm]{Lemma}
\theoremstyle{definition}

\newtheorem{remark}[thm]{Remark}
\newtheorem{example}[thm]{Example}

\newcommand{\N}{\mathbb{N}} 
\newcommand{\dx}{\partial_x}
\newcommand{\dt}{\partial_t}
\newcommand{\ds}{\partial_s}
\def\E{\mathcal{E}}
\def\V{\mathcal{V}}
\def\L{\mathcal{L}}
\def\G{\mathcal{G}}
\def\H{\mathcal{H}}
\def\M{\mathcal{M}}
\newcommand{\out}{_{\text{out}}}
\newcommand{\In}{_{\text{in}}} 
\newcommand{\rhor}{\rho_{\text{ref}}}
\newcommand{\lmax}{\ell_{\text{max}}}
\newcommand{\lmin}{\ell_{\text{min}}}

\title[Synchronization of a nodal observer for the wave equation]{Identification of minimal number of measurements allowing synchronization of a nodal observer for the wave equation}  

\author{Jan Giesselmann, Teresa Kunkel$^{\star}$}
\address{Technical University of Darmstadt, Department of Mathematics\\ Dolivostraße 15, 64293 Darmstadt, Germany}

\date{September 30, 2024}
\email{giesselmann@mathematik.tu-darmstadt.de}
\email{tkunkel@mathematik.tu-darmstadt.de}
\thanks{$^\star$ Corresponding author}

\begin{document}
	\maketitle
	
	\begin{abstract}  
		We study a state estimation problem for a $2\times 2$ linear hyperbolic system on networks with eigenvalues with opposite signs.
		The system can be seen as a simplified model for gas flow through gas networks.
		For this system we construct an observer system based on nodal measurements and investigate the convergence of the state of the observer system towards the original system state.
		We assume that measurements are available at the boundary nodes of the network
		and identify the minimal number of additional measurements in the network that are needed to guarantee synchronization of the observer state towards the original system state.
		It turns out that for tree-shaped networks boundary measurements suffice to guarantee exponential synchronization, while for networks that contain
		cycles synchronization can be guaranteed if and only if at least one measurement point is added in each cycle.
		This is shown for a system without source term and for a system with linear friction term.	
	\end{abstract}

\begin{quote}	
\end{quote}

\begin{quote}
\end{quote}

\section{Introduction}
In this paper, we consider a state estimation problem for a wave equation on networks that is written as a $2\times 2$ linear hyperbolic system with eigenvalues with opposite signs.
Our goal is to reconstruct the system state in the network from nodal measurements. Therefore we combine the model of the system with the measurements.
This is called data assimilation or state estimation. Here, we use observer-based data assimilation, which means that similar as in \cite{Gugat2021}
we construct an observer system whose boundary terms depend on the measurements.
Then, we investigate the convergence of the state of the observer system towards the original system state in the long time limit.
We refer to this convergence as synchronization. Note that Li and Rao, see e.g. \cite{LiRao2019}, use a different notion of synchronization.
They mean by synchronization that different components of a system admit (approximately) the same state in the long time limit. 

We assume that measurements are available at the boundary nodes of the network, but only very few measurements are available at points within the pipes.
This is motivated by the fact that the wave equation can be seen as a simplified model for gas transport through gas networks and in this application measurement devices are placed at the boundary of the network, while the number of measurement devices
at points within the network is very low. Therefore we assume that boundary measurements are available and identify the minimal number of additional measurements in the network that are needed to reconstruct the system state.
For tree-shaped networks, i.e., for networks that do not contain a cycle, we show that boundary measurements suffice to guarantee exponential convergence of the observer state towards the original system state in the long time limit.
For networks that contain a cycle, we show that in general the boundary measurements do not lead to synchronization of the observer state towards the original system state. However, if we add one measurement point in every cycle of the network, we again get synchronization, see Remark~\ref{rem:synchr_not_tree_shaped}.

The synchronization of the state of the observer system towards the state of the original system can be viewed as controlling the difference between these two systems to zero. Therefore, in particular for linear systems, convergence of observers is related to control and stabilization problems.
In the following, we give a literature review about boundary control and boundary stabilization problems for one-dimensional hyperbolic systems.
We start by systems on a single interval and then summarize results for systems on network.
In \cite{Gugat2014_telegraph-equation} the exponential stability of the telegraph equation on a bounded interval with boundary feedback control at one end of the interval is shown
and in \cite{BastinCoron2011}
conditions for boundary feedback stabilizability of a $2\times2$ linear hyperbolic system with feedback control at at least one end of a bounded interval are given, namely the system is stabilizable if the solution of an associated ODE exists.
The exponential stability of an observer-controller system for a system of coupled linear hyperbolic equations on a single edge	
is shown in \cite{MeglioVazquezKrstic2013}, where one boundary control point and one boundary measurement point is used.
In \cite{DatkoLagnesePlois1986} it is shown that a time delay in a boundary feedback control for a one-dimensional wave equation can destabilize the system. This might be relevant, if one assumes that there is a 
time delay for the measurements that are included in the boundary condition of the observer system.

An overview of control and stabilization of one-dimensional wave equations on networks, describing e.g. networks of linear vibrating strings,  can be found in \cite{Zuazua2013_control_stabilization}, see also references therein. The analysis in \cite{Zuazua2013_control_stabilization} focuses on controls or observations that take place at a single boundary node. 
In \cite{lagnese_modeling_1994} it is shown that tree-shaped networks of elastic strings are exactly controllable, if all but one boundary nodes are controlled, while the same equation on a networks that contains a cycle is not exactly controllable, if certain rationality conditions are satisfied.
Limits of stabilizability for networks of vibrating strings with feedback control at at least one boundary node are studied in \cite{GugatGerster2019} for star-shaped networks and it is proven that, depending on the parameters in the source term, stabilization is not possible, if the length of one of the edges is too long or the number of edges in the star-shaped network is too large.  
Similarly, in \cite{GugatHuangWang2023} for the same system on a network with a cycle it is shown that the system is only stabilizable, if the length of the edges is sufficiently small.
In \cite{Leugering2023_Control_WaveEquations, AvdoninZhao2022} the exact controllability of a linear wave equation on networks using boundary and internal controls is shown. 
The exponential stabilizability of a generalized telegraph equation on a star-shaped network with dissipative boundary conditions at all boundary nodes is proven in \cite{HayekNicaise2020}
and the exact controllability and exponential stability of $2\times2$ hyperbolic systems on graphs with dissipative boundary conditions at all but one boundary nodes is shown in \cite{Nicaise2017}. 
Estimates of the decay rate for a wave equation on star-shaped networks with boundary feedback control at all but one boundary nodes are given in \cite{XieXu2016}.

Stability and controllability has  also been studied for quasilinear systems.
Boundary controllability to a desired steady state for $2\times 2$-quasilinear systems of hyperbolic conservation laws on a bounded interval with boundary control at both ends of the interval is shown in \cite{CoronNovelGeorges2007}.
In \cite{GugatWeiland2021} it is shown that the solution of a hyperbolic quasilinear $2\times2$-system on networks can be stabilized exponentially fast by applying a sufficient number of damping control actions. The analysis includes an example of a network with a cycle. 
In \cite{PerrollazRosier2014} the finite-time stabilization of $2\times2$-systems of hyperbolic conservation laws is shown 
for tree-shaped networks with feedback control at at least one node per edge.
Quasilinear wave equations on a tree-shaped network of strings are investigated in \cite{Gu_TatsienLi2011} and exact boundary observability is shown, if the state at all but one boundary node is observed.
Further results on boundary control and boundary stabilization of one-dimensional hyperbolic systems can be found in the review paper \cite{Hayat2021} 
and references therein.

The main result of this paper is that we identify the minimal number of measurements within the network that are needed in addition to boundary measurements in order to guarantee synchronization of the observer state towards the original system state in the long time limit.
For tree-shaped networks, we show that boundary measurements suffice to guarantee synchronization, while for networks that contain cycles exactly one additional measurement point in every cycle is needed to guarantee synchronization.

We start by introducing the model equations and the observer system in Section~\ref{sec:model_observer}.
Then, in Section~\ref{sec:obs_no_friction} we study the synchronization of the observer system for a system without source term.
First, for networks containing a cycle we show that, if only boundary measurements are available, then there exist examples of initial data, such that the difference system has a constant solution, i.e., synchronization does not hold.
Second, for tree-shaped networks, we show that the state of the observer system converges exponentially towards the original system state in the long time limit. The main ingredient of the proof of the exponential synchronization is an observability inequality for tree-shaped networks, which can be proven using a graph-theoretic reduction argument over the size of the network. The reason to first present the case without source term is that the arguments are less technical in this case and to highlight 
that the exponential synchronization does not result from friction, but only from the boundary observations.
In Section~\ref{sec:obs_linear_friction} we extend the investigation of the synchronization to a system with linear source term.

\section{Model equations and observer system} \label{sec:model_observer}
Now, we introduce the model equations and the notation for networks.
We model a network by a directed graph $G=(\V, \E)$
with edges $e\in\E$ and nodes $\nu\in\V$. We denote the set of boundary nodes by $\V_\partial\subset \V$, the length of the edge $e\in\E$ by $\ell^e$ and the set of all edges that are incident to a node $\nu\in\V$ by $\E(\nu)$.
Let $\lmax$ denote the maximal length of an edge in $G$ and $\lmin$ denote the minimal length of an edge in $G$.
Then, on each edge we consider 
the following $2\times 2$ linear hyperbolic system
\begin{align} \label{eq:sys_single_pipe}
	\begin{pmatrix}
		\dt R_+^e \\ \dt R_-^e
	\end{pmatrix}
	+\begin{pmatrix}
		c & 0\\ 0 &-c
	\end{pmatrix}
	\begin{pmatrix}
		\dx R_+^e \\ \dx R_-^e
	\end{pmatrix}  
	= \begin{pmatrix}
		-\lambda (R_+^e-R_-^e)\\ \lambda (R_+^e-R_-^e)
	\end{pmatrix},\qquad
	x\in[0,\ell^e],\quad  t>0,
\end{align}
where $\lambda\ge0$ is some given parameter and $c>0$ denotes the speed of sound.
The system can also be written as the wave equation
\begin{align}
	\partial_{tt} u^e - c^2 \partial_{xx} u^e  = - 2\lambda\partial_t u^e, \label{eq:wave_eq}
\end{align}
where $u^e$ is related to $R^e_\pm$ by
\begin{align*}
	R^e_+=c\dx u^e -\dt u^e,\quad R^e_-=c\dx u^e +\dt u^e.
\end{align*}
One possible set of coupling conditions for the wave equation \eqref{eq:wave_eq} are the delta-prime vertex conditions
\begin{align*}
	&\dx u^e(t,\nu)=\dx u^f(t,\nu) ,&& t\in(0,T),\ \nu\in\V\setminus\V_\partial,\ e,f\in\E(\nu), \\ 
	&\sum_{e\in\E(\nu)} s^e(\nu) u^e(t,\nu) =0,&& t\in(0,T),\ \nu\in\V\setminus\V_\partial, 
\end{align*}
see \cite{Leugering2023_Control_WaveEquations}.
Here, $s^e(\nu)$ denotes the direction of the edge $e$. If $e=(\nu_1, \nu_2)$ starts in $\nu_1$ and ends in $\nu_2$, then $s^e(\nu_1)=-1$ and $s^e(\nu_2)=1$.
These coupling conditions imply the following coupling conditions for $R_\pm^e$
\begin{align*}
	&R^e_+(t,\nu)+R^e_-(t,\nu)=R^f_+(t,\nu)+R^f_-(t,\nu) ,&& t\in(0,T),\ \nu\in\V\setminus\V_\partial,\ e,f\in\E(\nu),\\ 
	&\sum_{e\in\E(\nu)} \left( R^e\out(t,\nu)-R^e\In(t,\nu) \right)=0,&& t\in(0,T),\ \nu\in\V\setminus\V_\partial 
\end{align*}
or equivalently
\begin{align}  \label{eq:CC} 
	R^e\out(t,\nu)=-R^e\In(t,\nu)+\tfrac{2}{|\E(\nu)|} \sum_{g\in\E(\nu)} R^g\In(t,\nu),\quad t\in(0,T),\ \nu\in\V\setminus\V_\partial,\ e\in\E(\nu).
\end{align}
Here we denote by $R^e\out(t,\nu)$ the variable $R^e_+(t,0)$, if $\nu$ corresponds to $x=0$, and $R^e_-(t,\ell^ e)$, if $\nu$ corresponds to $x=\ell^e$ (i.e., $R^e\out(t,\nu)$ is  directed out of the node $\nu$). Analogously, we denote by $R^e\In(t,\nu)$ the variable $R^e_-(t,0)$, if $\nu$ corresponds to $x=0$, and $R^e_+(t,\ell^ e)$, if $\nu$ corresponds to $x=\ell^e$.
In the following, we will use the coupling conditions \eqref{eq:CC} and 
complement the system with the initial and boundary conditions
\begin{align}
	&R^e_\pm(0,x)=y_\pm^e(x),\quad x\in(0,\ell^e),\ e\in\E, \label{eq:IC}\\
	&R^e\out(t,\nu)=b^e(t),\quad t\in(0,T),\ \nu\in\V_\partial,\ e\in\E(\nu). \label{eq:BC}
\end{align}

The system \eqref{eq:sys_single_pipe} together with the initial, boundary and coupling conditions \eqref{eq:CC}--\eqref{eq:BC} can be seen as a simplified model for the flow of gas through gas networks. In this case, $R_\pm$ correspond to the Riemann invariants of the barotropic Euler equations as a model for gas flow through pipes and $R_\pm$ can be related to the velocity $v$ and the density $\rho$ of the gas by 
\begin{align*}
	R_\pm=\int_{\rhor}^{\rho} \tfrac{\sqrt{p'(r)}}{r} dr\pm v,
\end{align*}
where $\rhor$ denotes a reference density and $p(\rho)$ the pressure law.
Then, the coupling conditions~\eqref{eq:CC} correspond to conservation of mass and continuity of pressure at the pipe junctions. For more details see \cite{Gugat2021}.
Note that in contrast to \cite{Gugat2021}, we use here a simplified, linear friction term at the right hand side of \eqref{eq:sys_single_pipe}
since this allows an easier analysis of the synchronization of the observer system.  In addition, we focus here on boundary measurements and allow only few measurements points within the networks, while \cite{Gugat2021} assumes that measurement data are available at a sufficient number of inner nodes (roughly at every second node).

In the following, we refer to \eqref{eq:sys_single_pipe} together with \eqref{eq:CC}--\eqref{eq:BC} as the \textbf{original system} and assume that measurements of $R_\pm$ at the boundary nodes are available.
Then, we construct the \textbf{observer system}
\begin{align}
	&\begin{pmatrix}
		\dt S_+^e \\ \dt S_-^e
	\end{pmatrix}
	+\begin{pmatrix}
		c & 0\\ 0 &-c
	\end{pmatrix}
	\begin{pmatrix}
		\dx S_+^e \\ \dx S_-^e
	\end{pmatrix}  
	= \begin{pmatrix}
		-\lambda (S_+^e-S_-^e)\\ \lambda (S_+^e-S_-^e)
	\end{pmatrix},&& e\in\E, \label{eq:obs_sys}\\
	&S^e_\pm(0,x)=z_\pm^e(x),&& x\in(0,\ell^e),\  e\in\E, \label{eq:obs_IC}\\
	&S^e\out(t,\nu)=b^e(t),&& t\in(0,T),\  \nu\in\V_\partial,\  e\in\E(\nu), \label{eq:obs_BC}\\
	&S^e\out(t,\nu)=-S^e\In(t,\nu)+\tfrac{2}{|\E(\nu)|} \sum_{g\in\E(\nu)} S^g\In(t,\nu),&& t\in(0,T),\  \nu\in\V\setminus\V_\partial,\  e\in\E(\nu).  \label{eq:obs_CC}
\end{align}
Note that the observer system has different initial data than the original system, since we do not know the exact initial data, but the same boundary data, since we assume no measurement errors.
Due to the linearity of the system the \textbf{difference system} for $\delta:=R-S$ has the same structure as the original system  and is given by
\begin{align*}
	&\begin{pmatrix}
		\dt \delta_+^e \\ \dt \delta_-^e
	\end{pmatrix}
	+\begin{pmatrix}
		c & 0\\ 0 &-c
	\end{pmatrix}
	\begin{pmatrix}
		\dx \delta_+^e \\ \dx \delta_-^e
	\end{pmatrix}  
	= \begin{pmatrix}
		-\lambda (\delta_+^e-\delta_-^e)\\ \lambda (\delta_+^e-\delta_-^e)
	\end{pmatrix}, 
	&& e\in\E, \\ 
	&\delta^e_\pm(0,x)=y_\pm^e(x)-z_\pm^e(x)=:\tilde y_\pm^e(x),&& x\in(0,\ell^e),\  e\in\E, \\
	&\delta^e\out(t,\nu)=0,&& t\in(0,T),\  \nu\in\V_\partial,\  e\in\E(\nu), \\ 
	&\delta^e\out(t,\nu)=-\delta^e\In(t,\nu)+\tfrac{2}{|\E(\nu)|} \sum_{g\in\E(\nu)} \delta^g\In(t,\nu),&& t\in(0,T),\  \nu\in\V\setminus\V_\partial,\  e\in\E(\nu).  
\end{align*}
In particular, the boundary data of the difference system is zero.

\section{Synchronization of observer for system without friction} \label{sec:obs_no_friction}
First, we investigate the synchronization of the boundary observer for $\lambda=0$, i.e., for the case without friction. Then, the difference system has the form 
\begin{align}
	&\begin{pmatrix}
		\dt \delta_+^e \\ \dt \delta_-^e
	\end{pmatrix}
	+\begin{pmatrix}
		c & 0\\ 0 &-c
	\end{pmatrix}
	\begin{pmatrix}
		\dx \delta_+^e \\ \dx \delta_-^e
	\end{pmatrix}  
	= 0,  && e\in\E, \label{eq:system_diff}\\
	&\delta^e_\pm(0,x)=\tilde y_\pm^e(x),&& x\in(0,\ell^e), e\in\E, \\
	&\delta^e\out(t,\nu)=0,&& t\in(0,T), \nu\in\V_\partial, e\in\E(\nu), \label{eq:diff_BC} \\
	&\delta^e\out(t,\nu)=-\delta^e\In(t,\nu)+\tfrac{2}{|\E(\nu)|} \sum_{g\in\E(\nu)} \delta^g\In(t,\nu),&& t\in(0,T), \nu\in\V\setminus\V_\partial.  \label{eq:diff_CC}
\end{align}
Since we have no source term, the variables $\delta_\pm$ are constant along the characteristic curves $(s, \xi_\pm(s,x,t))=(s, x\pm c(s-t
))$, i.e., $\delta^e_+(t, \ell^e)=\delta^e_+(t-\tfrac{\ell^e}{c}, 0)$.
Existence of solutions of \eqref{eq:system_diff}--\eqref{eq:diff_CC} can be shown by tracing $\delta_\pm$ back along the characteristic curves.

\subsection{Example: No synchronization in a network with a cycle} \label{sec:example_no_synchr}
In this section we present an example that shows that in general it is not possible to guarantee synchronization of the observer system towards the original system, if the network contains a cycle and only boundary measurements are given.
We consider a network that contains a cycle, where the cycle does not contain a boundary node, see Fig.~\ref{fig:network_with_cycle}.

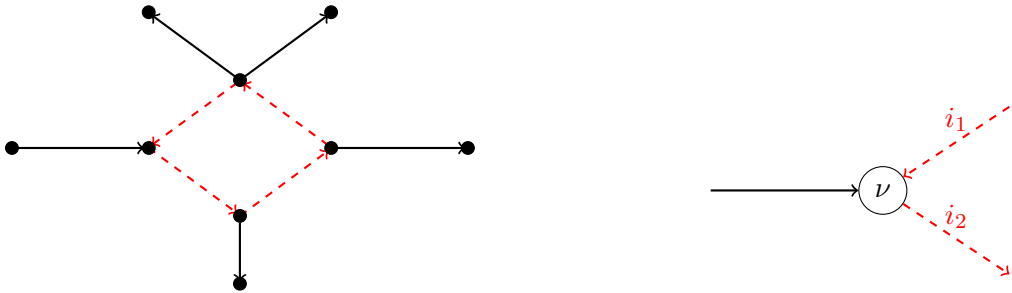
\begin{figure}[ht]
	\centering
	\begin{subfigure}[b]{0.49\textwidth}
		\centering
		\begin{tikzpicture}[scale=1.2]
			
			\draw[black, thick, ->] (3,0) coordinate (A)  -- (4.45,0); 
			\draw[black, thick, <-, red, dashed] (4.53,0.03)   -- (5.47,0.72);
			\draw[black, thick, ->, red, dashed] (4.5,0) coordinate (B)-- (5.47,-0.72);
			\draw[black, thick, ->] (5.5,0.75) coordinate (C)-- (6.47,1.47);
			\draw[black, thick, <-, red, dashed] (5.53,0.72) -- (6.47,0.03);
			\draw[black, thick, ->] (C) -- (4.53,1.47);
			\draw[black, thick, ->] (5.5,-0.75) coordinate (D)-- (5.5,-1.47);
			\draw[black, thick, ->, red, dashed] (D)-- (6.47,-0.03);
			\draw[black, thick, ->] (6.5,0) coordinate (E)-- (7.97,0);

			\filldraw [black] (A) circle (2pt);
			\filldraw [black] (B)  circle (2pt) ;
			\filldraw [black] (C) circle (2pt);
			\filldraw [black] (D) circle (2pt);
			\filldraw [black] (E) circle (2pt);
			\filldraw [black] (6.5, 1.5) circle (2pt);
			\filldraw [black] (4.5, 1.5) circle (2pt);
			\filldraw [black] (8.0, 0) circle (2pt);
			\filldraw [black] (5.5, -1.5) circle (2pt);
		\end{tikzpicture}
	\end{subfigure}
	\hfill
	\begin{subfigure}[b]{0.49\textwidth}
		\centering
		\begin{tikzpicture}[scale=1.2]
			
		\node (1) at (-2,0){};
		\node (2) at (0,0) [circle,draw]{$\nu$};
		\node (3) at (1.5,-1) {};
		\node (4) at (1.5,1){};
		
		\draw[->, thick] (1) to node[above] {} (2);
		\draw[->, thick, red, dashed] (2) to node[above] {$i_2$} (3);
		\draw[->, thick, red, dashed] (4) to node[above] {$i_1$} (2);
		\end{tikzpicture}
	\end{subfigure}
	
	\caption{Example of a network with a cycle (left) and a node within the cycle with incident edges (right). The edges contained in the cycle are drawn in red dashed. For simplicity we assume that all edges in the cycle are oriented in the same direction with respect to the cycle, e.g. counter-clockwise.
	}
	\label{fig:network_with_cycle}
\end{figure}

Denote the set of edges that are contained in the cycle by $\E_c$ and the set of edges that are not contained in the cycle by $\E_n$.
Now assume that at the initial time $t=0$ for all edges in the cycle we have the initial condition
\begin{align}
	\delta_+^e(0,x)=a, \quad \delta_-^e(0,x)=-a,\quad e\in\E_c ,\ x\in(0,\ell^e)  \label{eq:exl_no_synchr-IC1}
\end{align}
for some constant $a>0$ and for all edges that are not contained in the cycle we have 
\begin{align}
	\delta_\pm^e(0,x)=0, \quad e\in\E_n ,\ x\in(0,\ell^e). \label{eq:exl_no_synchr-IC2}
\end{align} 
This initial condition is compatible with the boundary conditions \eqref{eq:diff_BC} and the coupling conditions \eqref{eq:diff_CC}. To see this, let $\nu\in\V$ be a node that is contained in the cycle and denote by $i_1, i_2$ the edges that are incident to $\nu$ and are contained in the cycle (see Fig.~\ref{fig:network_with_cycle}). Then we have
\begin{align*}
	&-\delta^{i_2}\In(0,\nu)+\tfrac{2}{|\E(\nu)|} \sum_{g\in\E(\nu)} \delta^g\In(0,\nu)\\
	&=-\delta^{i_2}\In(0,\nu)+\tfrac{2}{|\E(\nu)|} \left(\delta^{i_1}\In(0,\nu) +\delta^{i_2}\In(0,\nu)\right)\\
	&= -(-a) + \tfrac{2}{|\E(\nu)|} (a+(-a)) = a =\delta^{i_2}\out(0,\nu).
\end{align*}
Therefore \eqref{eq:system_diff}--\eqref{eq:diff_CC} admits a constant solution 
\begin{align*}
	&\delta_+^e(t,x)=a, \quad \delta_-^e(t,x)=-a,\quad e\in\E_c,\ t>0,\ x\in(0,\ell^e)\\
	&\delta_\pm^e(t,x)=0, \quad e\in\E_n,\ t>0,\ x\in(0,\ell^e)
\end{align*} 
that does not decay over time, i.e., the solution of the  observer system does not converge towards the original solution.
In particular, we have on all pipes that are contained in the cycle that $\delta_+=a$ and $\delta_-=-a$, which corresponds to $R_+=S_++a$, $R_-=S_--a$. If \eqref{eq:system_diff}--\eqref{eq:diff_CC} is viewed as a model for gas transport, this corresponds to the fact that
the state of the original system and the state of the observer system have the same density, but different velocities.

\subsection{Observability inequality for tree-shaped networks}
As a first step towards the exponential synchronization of the observer state towards the original system state, we show a $L^2$-observability inequality for tree-shaped networks. 
The proof of the observability inequality uses a graph-theoretic reduction argument.

\begin{lem}  \label{lem:ObsInequ_L2}
	Let $G=(\V, \E)$ be a tree-shaped network with $N$ inner nodes. Let $\lmax$ denote the maximal length of a pipe in $G$. Then there exists a constant $C>0$ such that for $T\ge N \tfrac{\lmax}{c}$ and $t>T$ we have
	\begin{align} \label{eq:obsInequ_L2}
		\|(\delta_+, \delta_-)(t,\cdot)\|^2_{L^2(\E)}
	\le C \sum_{\nu\in\V_\partial}  \|(\delta_+, \delta_-)(\cdot, \nu)\|^2_{L^2([t-T, t+T])},
	\end{align}
	where
	\begin{align*}
		\|(\delta_+, \delta_-)(t,\cdot)\|^2_{L^2(\E)}
		:=\|\delta_+(t,\cdot)\|^2_{L^2(\E)} + \|\delta_-(t,\cdot)\|^2_{L^2(\E)}.
	\end{align*}
	The constant $C$ depends on the topology of the network and on $c$.
\end{lem}

\begin{proof}
	The main idea of the proof is to use an induction argument over the number of inner nodes of the network.
	In order to define the induction step, we first have to introduce some notation.
	In the following, we denote all edges that have a boundary node as incident node as `boundary edges' and all edges, where both incident nodes are inner nodes, by `inner edges'.
	By `boundary node of first order' we denote all all inner nodes of $G$ that are incident to a boundary edge.
	
	As a first step of the proof we show the following assertion: If $G$ is a tree-shaped network with $N\ge 2$ inner nodes, then there exists an inner node $\nu_1$ of $G$ such that $\nu_1$ is a boundary node of first order of $G$ and exactly one of the incident edges of $\nu_1$ is an inner edge, while all other incident edges of $\nu_1$ are boundary edges (see Fig.~\ref{fig:network2_tree-shaped_reduction}). The existence of such a node $\nu_1$ will be the basis of the reduction step used for the induction.

		\begin{figure}[ht]
		\centering
		\begin{tikzpicture}[scale=1.2]
			
			\draw[gray, thick, ->, dashed] (2,-0.75) -- (2.97,-0.03);
			\draw[gray, thick, ->, dashed] (2,0.75) -- (2.97,0.03);
			\draw[black, thick, ->] (4.5,0) -- (5.47,0.72);
			\draw[black, thick, ->] (3,0) node[above right] {$\nu_1$} -- (4.45,0);
			\draw[black, thick, ->] (4.5,0) -- (5.47,-0.72);
			\draw[black, thick, ->] (5.5,0.75) -- (6.47,1.47);
			\draw[black, thick, ->] (5.5,0.75) -- (6.47,0.23);
			\draw[black, thick, ->] (5.5,0.75) -- (6.47,-0.47);
			\draw[black, thick, ->] (5.5,-0.75) -- (6.47,-1.47);

			\filldraw [black] (3,0) circle (2pt);
			\filldraw [black] (4.5,0) circle (2pt);
			\filldraw [gray] (2,0.75) circle (2pt);
			\filldraw [gray] (2,-0.75) circle (2pt);
			\filldraw [black] (5.5,0.75) circle (2pt);
			\filldraw [black] (5.5, -0.75) circle (2pt);
			\filldraw [black] (6.5, 1.5) circle (2pt);
			\filldraw [black] (6.5, 0.2) circle (2pt);
			\filldraw [black] (6.5, -0.5) circle (2pt);
			\filldraw [black] (6.5, -1.5) circle (2pt);
		\end{tikzpicture}
		\caption{Illustration of the reduction step and of the definition of the node~$\nu_1$. The network $G$ consists of all (dashed and solid) edges, while the reduced network $G_1$ only consists of the solid edges.}
		\label{fig:network2_tree-shaped_reduction}
	\end{figure}
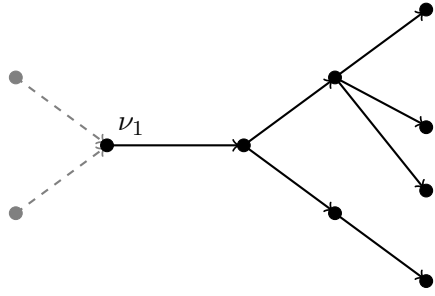

	We now show that we can always find a node $\nu_1$ with the above mentioned properties as long as the current network has at least two inner nodes. This can be reduced to a question from graph theory:
	Consider the subnetwork $\tilde G$ of $G$ that is constructed by removing all boundary edges of $G$, i.e., the nodes of $\tilde G$ are all inner nodes of $G$ (see Fig. \ref{fig:network3_tree-shaped_tildeG}).
	Since $G$ has at least two inner nodes and $G$ is a tree, the subnetwork $\tilde G$ is a tree with at least two nodes. A standard theorem from graph theory
	(see e.g.  \cite[Theorem 3.4]{saoub2021_introduction_graph_theory})
	states that every tree with at least two nodes has at least one node of degree one, i.e., at least one node that has exactly one incident edge. 
	Therefore $\tilde G$ 
	has at least one node $\tilde \nu$ that has exactly one incident edge in $\tilde G$. 
	Since $\tilde \nu$ is an inner node of $G$, $\tilde \nu$ has at least two incident edges in $G$. If none of these edges was a boundary edge of $G$, then all of these edges would be in $\tilde G$, 
	which would be a contradiction to the fact that $\tilde \nu$ is a node of order 1 in $\tilde G$.
	Hence, at least one of these edges is a boundary edge of $G$, i.e., $\tilde \nu$ is a boundary node of first order of $G$. Since $\tilde \nu$ has degree one in $\tilde G$, $\tilde \nu$ has exactly one incident inner edge in $G$. In particular, we can define $\nu_1:=\tilde \nu$ and $\nu_1$ has the required properties.
	A similar reduction argument (but for a different system) is used in \cite[Chapter II.5]{lagnese_modeling_1994}.

		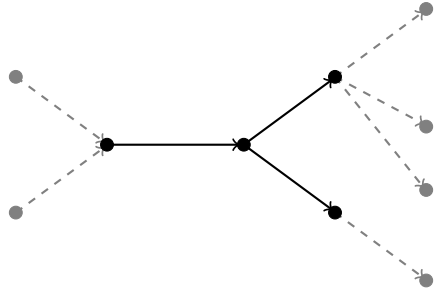
\begin{figure}[ht]
		\centering
		\begin{tikzpicture}[scale=1.2]
			
			\draw[gray, thick, ->, dashed] (2,-0.75) -- (2.97,-0.03);
			\draw[gray, thick, ->, dashed] (2,0.75) -- (2.97,0.03);
			\draw[black, thick, ->] (4.5,0) -- (5.47,0.72);
			\draw[black, thick, ->] (3,0)  -- (4.45,0);
			\draw[black, thick, ->] (4.5,0) -- (5.47,-0.72);
			\draw[gray, thick, ->, dashed] (5.5,0.75) -- (6.47,1.47);
			\draw[gray, thick, ->, dashed] (5.5,0.75) -- (6.47,0.23);
			\draw[gray, thick, ->, dashed] (5.5,0.75) -- (6.47,-0.47);
			\draw[gray, thick, ->, dashed] (5.5,-0.75) -- (6.47,-1.47);

			\filldraw [black] (3,0) circle (2pt);
			\filldraw [black] (4.5,0) circle (2pt);
			\filldraw [gray] (2,0.75) circle (2pt);
			\filldraw [gray] (2,-0.75) circle (2pt);
			\filldraw [black] (5.5,0.75) circle (2pt);
			\filldraw [black] (5.5, -0.75) circle (2pt);
			\filldraw [gray] (6.5, 1.5) circle (2pt);
			\filldraw [gray] (6.5, 0.2) circle (2pt);
			\filldraw [gray] (6.5, -0.5) circle (2pt);
			\filldraw [gray] (6.5, -1.5) circle (2pt);
		\end{tikzpicture}
		\caption{Construction of the auxiliary network $\tilde G$ (solid lines) out of the network $G$ (dashed and solid lines).}
		\label{fig:network3_tree-shaped_tildeG}
	\end{figure}
	
	Now, we define the following reduction step:
	Assume that the observability inequality \eqref{eq:obsInequ_L2} holds for all tree-shaped networks with at most $N-1$ inner nodes (we assume here $N\ge2$, otherwise we can directly apply the base case $N=1$, see the end of this proof). By the first part of the proof there exists an inner node $\nu_1$ of $G$ that is a boundary node of first order of $G$ and that has exactly one incident inner edge.
	Now, we consider the subnetwork $G_1$ of $G$ that is constructed by removing all incident boundary edges of $\nu_1$ (cf. Fig. \ref{fig:network2_tree-shaped_reduction}).
	Then, $G_1$ is a tree-shaped network with $N-1$ inner nodes and by assumption there exists a constant $C_1>0$ such that
	\begin{align} \label{eq:ObsInequ_L2_G1}
		\|(\delta_+, \delta_-)(t,\cdot)\|^2_{L^2(\E(G_1))}
		\le C_1 \sum_{\nu\in\V_\partial(G_1)}  \|(\delta_+, \delta_-)(\cdot, \nu)\|^2_{L^2([t-(N-1) \tfrac{\lmax}{c},\, t+(N-1) \tfrac{\lmax}{c}])}.
	\end{align}
	Our goal is to induce \eqref{eq:obsInequ_L2} from \eqref{eq:ObsInequ_L2_G1}.
	Since $\V_\partial(G_1)\subset\{\nu_1\}\cup \V_\partial(G)$, as a first step we have to estimate 
	$ \|(\delta_+^e, \delta_-^e)(\cdot, \nu_1)\|^2_{L^2([t-(N-1) \tfrac{\lmax}{c}, \, t+(N-1) \tfrac{\lmax}{c}])} $ by the right hand side of \eqref{eq:obsInequ_L2}, where $e$ is the inner edge that is incident to $\nu_1$.
	Denote the boundary edges that are incident to $\nu_1$ in $G$ by $e=1,\ldots, m$
	and the inner edge that is incident to $\nu_1$ by $e=m+1$, see Fig.~\ref{fig:network_coupling_condition}.
	The coupling conditions \eqref{eq:diff_CC} imply
	\begin{align*}
		\delta^{m+1}\out(s,\nu_1)&=(-1+\tfrac{2}{|\E(\nu_1)|})\, \delta^{m+1}\In(s,\nu_1)+\tfrac{2}{|\E(\nu_1)|} \sum_{i=1}^{m} \delta^i\In(s,\nu_1), \\
		\delta^{m+1}\In(s,\nu_1)&=\tfrac{|\E(\nu_1)|}{2} \left(  \delta^j\out(s,\nu_1)+\delta^j\In(s,\nu_1) \right)
		- \sum_{i=1}^{m} \delta^i\In(s,\nu_1) \quad \forall j=1,\ldots, m
	\end{align*}
	with $|\E(\nu_1)|=m+1$. Now, we can estimate
	\begin{align*}
		&|\delta^{m+1}\out(s,\nu_1)|^2 + |\delta^{m+1}\In(s,\nu_1)|^2\\
		&\le 2\, |\delta^{m+1}\In(s,\nu_1)|^2 + 2\, \big| \sum_{i=1}^{m} \delta^i\In(s,\nu_1) \big|^2
			+ |\delta^{m+1}\In(s,\nu_1)|^2\\
		&\le  3 \left( \tfrac{(m+1)^2}{2} |\delta^j\out(s,\nu_1)+\delta^j\In(s,\nu_1)|^2 + 2  |\sum_{i=1}^{m} \delta^i\In(s,\nu_1) |^2\right)     
			+ 2\, \big| \sum_{i=1}^{m} \delta^i\In(s,\nu_1) \big|^2\\
		&\le \tfrac{3}{2} (m+1)^2 |\delta^j\out(s,\nu_1)+\delta^j\In(s,\nu_1)|^2 
			+ 8 m \sum_{i=1}^{m} |\delta^i\In(s,\nu_1)|^2\\
	\end{align*}
	with an arbitrary $j\in\{1,\ldots,m\}$, where we have used Jensen's inequality in the last step.
	If $M$ is the maximal number of edges that are incident to any node in $G$, then 
	\begin{align} \label{eq:proofL2-obs-ineq_1}
		|\delta^{m+1}\out(s,\nu_1)|^2 + |\delta^{m+1}\In(s,\nu_1)|^2
		\le \left(\tfrac{3}{2} M^2 +8 (M-1)\right) \sum_{i=1}^{m} \left(|\delta^i\In(s,\nu_1)|^2 +|\delta^i\out(s,\nu_1)|^2 \right).
	\end{align}

	\begin{figure}[ht]
		\centering
		\begin{tikzpicture}[scale=1.2]
			
			\draw[black, thick, ->] (2.3,-0.85) coordinate (A) node[left] {$\nu_{\partial_m}$}--  node[below right] {\small $m$} (2.97,-0.03);
			\draw[black, thick, ->] (2,0.75) node[left] {$\nu_{\partial_1}$} -- node[above] {\small $1$} (2.97,0.03);
			\draw[black, thick, ->] (1.7,0) node[left] {$\nu_{\partial_2}$} -- node[below] {\small $2$} (2.97,0.0);
			\draw[black, thick, dashed] (4.5,0) -- (5.47,0.72);
			\draw[black, thick, ->] (3,0) node[above right]{$\nu_1$}  -- node[below]{\small $m+1$} (4.45,0);
			\draw[black, thick, dashed] (4.5,0) -- (5.47,-0.72);

			\filldraw [black] (3,0) circle (2pt);
			\filldraw [black] (4.5,0) circle (2pt);
			\filldraw [black] (2,0.75) circle (2pt);
			\filldraw [black] (A) circle (2pt);
			\filldraw [black] (1.7,0) circle (2pt);
		\end{tikzpicture}
		\caption{Illustration of the notation for the coupling condition at the node $\nu_1$.}
		\label{fig:network_coupling_condition}
	\end{figure}
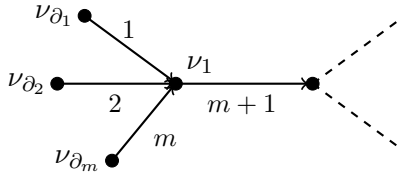

	Let $\nu_{\partial_j}$ denote the boundary node that is incident to the edge $j$, $j\in\{1,\ldots,m\}$, in $G$. Then, since the value of $\delta_\pm$ is constant along the characteristic curves, we have
	\begin{align*}
		\delta^j\out(s,\nu_1) = \delta^j\In(s+\tfrac{\ell^j}{c},\nu_{\partial_j}),\quad
		\delta^j\In(s,\nu_1) = \delta^j\out(s-\tfrac{\ell^j}{c},\nu_{\partial_j}).
	\end{align*}
	This implies
	\begin{align*}
		\int_{t-(N-1)\tfrac{\lmax}{c}}^{t+(N-1)\tfrac{\lmax}{c}} |\delta^j\out(s,\nu_1)|^2 ds
		=\int_{t-(N-1)\tfrac{\lmax}{c}+\tfrac{\ell^j}{c}}^{t+(N-1)\tfrac{\lmax}{c}+\tfrac{\ell^j}{c}} |\delta^j\In(s,\nu_{\partial_j})|^2 ds
		\le \int_{t-N\tfrac{\lmax}{c}}^{t+N\tfrac{\lmax}{c}} |\delta^j\In(s,\nu_{\partial_j})|^2 ds
	\end{align*}
	and analogously
	\begin{align*}
		\int_{t-(N-1)\tfrac{\lmax}{c}}^{t+(N-1)\tfrac{\lmax}{c}} |\delta^j\In(s,\nu_1)|^2 ds
		\le \int_{t-N\tfrac{\lmax}{c}}^{t+N\tfrac{\lmax}{c}} |\delta^j\out(s,\nu_{\partial_j})|^2 ds
	\end{align*}
	for $j\in\{1,\ldots,m\}$.
	Combining this with \eqref{eq:proofL2-obs-ineq_1} yields
	\begin{align*}
		&\|(\delta_+^{m+1}, \delta_-^{m+1})(\cdot, \nu_1)\|^2_{L^2([t-(N-1) \tfrac{\lmax}{c}, \, t+(N-1) \tfrac{\lmax}{c}])}\\
		&\le \int_{t+(N-1)\tfrac{\lmax}{c}}^{t-(N-1)\tfrac{\lmax}{c}} \left(\tfrac{3}{2} M^2 +8 (M-1)\right) \sum_{i=1}^{m} \left(|\delta^i\In(s,\nu_1)|^2 +|\delta^i\out(s,\nu_1)|^2 \right) ds\\
		&\le \left(\tfrac{3}{2} M^2 +8 (M-1)\right) \int_{t-N\tfrac{\lmax}{c}}^{t+N\tfrac{\lmax}{c}}\sum_{i=1}^{m} \left(|\delta^i\In(s,\nu_{\partial_i})|^2 +|\delta^i\out(s,\nu_{\partial_i})|^2 \right) ds.
	\end{align*}
	This concludes the first step of the proof of \eqref{eq:obsInequ_L2} using \eqref{eq:ObsInequ_L2_G1}.
	
	As a second step, we estimate $\|(\delta_+^j, \delta_-^j)(t,\cdot)\|^2_{L^2(0,\ell^j)}$ for all removed boundary edges\linebreak ${j=1,\ldots,m}$. Again, using the characteristic curves, we can compute
	\begin{align} \label{eq:proofL2-obs-ineq_2}
		\int_0^{\ell^j} |\delta\out^j(t,x)|^2 dx
		= \int_0^{\ell^j} |\delta^j\In(t+\tfrac{|\nu_{\partial_j}-x|}{c},\nu_{\partial_j})|^2 dx
		= \int_t^{t+\tfrac{\ell^j}{c}} |\delta^j\In(s,\nu_{\partial_j})|^2 c\, ds.
	\end{align}
	This implies
	\begin{align*}
		\|(\delta_+^j, \delta_-^j)(t,\cdot)\|^2_{L^2(0,\ell^j)}
		\le c\, \|(\delta_+^j, \delta_-^j)(\cdot, \nu_{\partial_j})\|^2_{L^2([t- \tfrac{\lmax}{c}, \, t+ \tfrac{\lmax}{c}])}
	\end{align*}
	for $j\in\{1,\ldots,m\}$.
	
	In summary, we have shown that the induction hypothesis \eqref{eq:ObsInequ_L2_G1} implies
	\begin{align*}
		&\|(\delta_+, \delta_-)(t,\cdot)\|^2_{L^2(\E(G))}\\
		&\le c\, \sum_{i=1}^{m} \|(\delta_+^j, \delta_-^j)(\cdot, \nu_{\partial_j})\|^2_{L^2([t- \tfrac{\lmax}{c}, \, t+ \tfrac{\lmax}{c}])}\\
		&\quad +C_1 ( \tfrac{3}{2} M^2 +8 (M-1) ) \sum_{\nu\in\V_\partial(G)}  \|(\delta_+, \delta_-)(\cdot, \nu)\|^2_{L^2([t-N \tfrac{\lmax}{c}, \, t+N \tfrac{\lmax}{c}])}\\
		&\le C \sum_{\nu\in\V_\partial(G)}  \|(\delta_+, \delta_-)(\cdot, \nu)\|^2_{L^2([t-N \tfrac{\lmax}{c}, \, t+N \tfrac{\lmax}{c}])},
	\end{align*}
	i.e, we have shown that \eqref{eq:ObsInequ_L2_G1} implies \eqref{eq:obsInequ_L2}.
	
	It remains to show the base case of the induction. Consider a tree-shaped network $G_{N-1}$ with exactly one inner node. Then this network is a star-shaped network, cf. Fig.~\ref{fig:star-shaped_network}. Denote the inner node by $\tilde \nu$, the edges by $e=1,\ldots, \tilde m$ and the boundary nodes by $\nu_{\partial_j}$, $j=1,\ldots, \tilde m$.
	As in \eqref{eq:proofL2-obs-ineq_2} we can show
	\begin{align*}
		&\|(\delta_+, \delta_-)(t,\cdot)\|^2_{L^2(\E(G_{N-1}))}
		\le c \sum_{i=1}^{\tilde m} \int_{t-\tfrac{\lmax}{c}}^{t+\tfrac{\lmax}{c}}
		\left(|\delta^i\In(s,\nu_{\partial_i})|^2 +|\delta^i\out(s,\nu_{\partial_i})|^2 \right) ds \\
		&	= c \sum_{\nu\in\V_\partial(G_{N-1})}  \|(\delta_+, \delta_-)(\cdot, \nu)\|^2_{L^2([t- \tfrac{\lmax}{c}, \, t+\tfrac{\lmax}{c}])}.
	\end{align*}
	Together, this shows the assertion of the lemma.
\end{proof}

\begin{figure}[ht]
	\centering
	\begin{tikzpicture}[scale=1.2]
		
		\draw[black, thick] (2,-0.75) -- (2.97,-0.03);
		\draw[black, thick] (2,0.75) -- (2.97,0.03);
		\draw[black, thick] (3,0) node[above] {$\tilde \nu$} --  (4.4,0.5) coordinate (A);
		\draw[black, thick] (3,0)  --  (4.4,-0.5) coordinate (B);

		\filldraw [black] (3,0) circle (2pt) ;
		\filldraw [black] (A) circle (2pt) node[right] {$\nu_{\partial_3}$};
		\filldraw [black] (B) circle (2pt) node[right] {$\nu_{\partial_4}$};
		\filldraw [black] (2,0.75) circle (2pt) node[left] {$\nu_{\partial_1}$};
		\filldraw [black] (2,-0.75) circle (2pt) node[left] {$\nu_{\partial_2}$};
	\end{tikzpicture}
	\caption{Example of a star-shaped network. }
	\label{fig:star-shaped_network}
\end{figure}
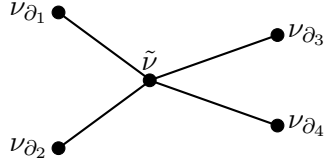

\subsection{Exponential synchronization}
Based on the $L^2$-observability inequality stated in \linebreak
Lemma~\ref{lem:ObsInequ_L2},
we can show that the state of the observer system converges exponentially towards the original solution for tree-shaped networks. The proof uses the same strategy as the proof of \cite[Theorem 4]{Gugat2021}.

\begin{thm}  \label{thm:expSynchr_L2}
	Let $G=(\V, \E)$ be a tree-shaped network. Then there exist constants $\mu>0$, $C_1>0$ depending on $c$, the topology of the network, the length of the pipes and the size of the initial data
	such that 
	\begin{align*}
		\|(\delta_+, \delta_-)(t,\cdot)\|^2_{L^2(\E)}
		\le C_1 \exp(-\mu t) \quad \forall t>0.
	\end{align*}
\end{thm}

\begin{proof}
	Let some $t>0$ be given. As in the proof of \cite[Theorem 4]{Gugat2021}, we multiply \eqref{eq:system_diff} by $\delta_\pm^e$ and integrate over $[t-\tilde t, t+\tilde t]\times[0,\ell^e]$ for some $\tilde t\in(0,t)$.
	This yields
	\begin{align*}
		\int_{t-\tilde t}^{t+\tilde t} \int_{0}^{\ell^e} \delta_\pm^e(s,x) \dt \delta_\pm^e(s,x)\, dx\, ds
		=(\mp c) \int_{t-\tilde t}^{t+\tilde t} \int_{0}^{\ell^e} \delta_\pm^e(s,x) \dx \delta_\pm^e(s,x)\, dx\, ds,
	\end{align*}
	i.e.,
	\begin{align*}
		\tfrac{1}{2}  \int_{0}^{\ell^e} \left[  | \delta_\pm^e(s,x)|^2 \right]_{s=t-\tilde t}^{t+\tilde t} \, dx
		= (\mp \tfrac{c}{2})  \int_{t-\tilde t}^{t+\tilde t} \left[  | \delta_\pm^e(s,x)|^2 \right]_{x=0}^{\ell^e} \, ds.
	\end{align*}
	Summing up over all pipes $e\in\E$, we get the following equality for the $L^2$-norm
	\begin{align*}
		&\|(\delta_+, \delta_-)(t+\tilde t,\cdot)\|^2_{L^2(\E)} - \|(\delta_+, \delta_-)(t-\tilde t,\cdot)\|^2_{L^2(\E)}\\
		&= (-c) \sum_{e\in\E} \int_{t-\tilde t}^{t+\tilde t} \left[  | \delta_+^e(s,x)|^2  -| \delta_-^e(s,x)|^2 \right]_{x=0}^{\ell^e} \, ds\\
		&= c  \int_{t-\tilde t}^{t+\tilde t} \sum_{\nu\in\V} \sum_{e\in\E(\nu)}   \big( |\delta\out^e(s,\nu)|^2  -| \delta\In^e(s,\nu)|^2 \big)\, ds.
	\end{align*}
	For all inner nodes $\nu$ of the network, \cite[Lemma 2]{Gugat2021} implies that 
	\begin{align*}
		\sum_{e\in\E(\nu)}   \big( |\delta\out^e(s,x)|^2  -| \delta\In^e(s,x)|^2 \big)=0
	\end{align*}
	 (note that in the notation of \cite[Lemma 2]{Gugat2021} we have $\mu^\nu=1$ at all inner nodes).
	 For all boundary nodes $\nu\in\V_\partial$, the boundary condition \eqref{eq:diff_BC} implies $|\delta\out^e(s,\nu)|^2 =0$. Together this shows
	 \begin{align*}
	 	&\|(\delta_+, \delta_-)(t+\tilde t,\cdot)\|^2_{L^2(\E)} - \|(\delta_+, \delta_-)(t-\tilde t,\cdot)\|^2_{L^2(\E)}\\
	 	&\le (-c)  \int_{t-\tilde t}^{t+\tilde t} \sum_{\nu\in\V_\partial} \sum_{e\in\E(\nu)}   \big( |\delta\out^e(s,\nu)|^2  +| \delta\In^e(s,\nu)|^2 \big)\, ds\\
	 	&= (-c) \sum_{\nu\in\V_\partial}  \|(\delta_+, \delta_-)(\cdot, \nu)\|^2_{L^2([t-\tilde t, t+\tilde t])}.
	 \end{align*}
	If we denote
	\begin{align*}
		\L(t):=\|(\delta_+, \delta_-)(t,\cdot)\|^2_{L^2(\E)},
	\end{align*} 
	then this can be written as
	\begin{align*}
		\L(t+\tilde t)-\L(t-\tilde t)\le (-c) \sum_{\nu\in\V_\partial}  \|(\delta_+, \delta_-)(\cdot, \nu)\|^2_{L^2([t-\tilde t, t+\tilde t])}.
	\end{align*}
 	Since $\tilde t\in(0,t)$ was chosen arbitrarily, a first consequence of this estimate is that 
 	$\L(t)$ is monotonically decreasing in $t$.
 	Now, assume that $\tilde t\ge T :=N \tfrac{\lmax}{c}$ and $t>T$, where $N$ is the  number of inner nodes of the network $G$ and $\lmax$ is the maximal length of a pipe in $G$. Then, we can apply the observability inequality \eqref{eq:obsInequ_L2} in order to show
 	\begin{align*} 
 		\L(t+T)-  \L(t-T)  \le (-\tfrac{c}{C}) \L(t),
 	\end{align*}
 	where $C>0$ is the constant in the observability inequality \eqref{eq:obsInequ_L2}.
	Since $\L(t)$ is decreasing, this implies
	\begin{align} \label{eq:proof_expSynchr_2}
		\L(t+T)\le \tilde C \L(t-T)
	\end{align}
	with $\tilde C :=\tfrac{1}{1+\tfrac{c}{C}}<1$.
	
	Now, we use a modified Gronwall-inequality (cf. \cite[Lemma 1]{Gugat2010}) in order to show that \eqref{eq:proof_expSynchr_2} implies exponential decrease of $\L(t)$:
	Let $s>0$. Then $s$ can be uniquely written as $s=\tilde s + 2 j T$ for some $j\in\N_0$ and $\tilde s\in[0,2T)$.
	By \eqref{eq:proof_expSynchr_2} we know that
	\begin{align*}
		\L(s) = \L(\tilde s + 2 j T) \le \tilde C ^j \L(\tilde s) \le  \tilde C ^j \L(0).
	\end{align*}
	Define $\mu :=-\tfrac{\ln(\tilde C)}{2T}$ and $C_1:=\exp(2T\mu)\L(0)$. Then $\mu>0$ due to $\tilde C<1$ and we have
	\begin{equation*}
		\L(s) \le \tilde C^j \L(0) = \exp(-2jT \mu) \L(0) 
		= \exp(-(2jT+\tilde s) \mu)  \exp(\tilde s \mu) \L(0) 
		\le \exp(-\mu s) C_1.  \qedhere
	\end{equation*}
\end{proof}

\begin{remark} \label{rem:synchr_not_tree_shaped}
	We can also show synchronization for general networks, if we add one suitable additional measurement point for every cycle of the network.
	More precisely, let $G$ be a network that contains cycles. For every cycle choose one edge $e$ that is part of the cycle and add at one point $x$ in this edge a measurement device that can measure the complete state of the gas (e.g. the value of $v$ and $\rho$, from which we can compute $R_+$ and $R_-$).  
	Mathematically this can be interpreted as if we cut the edge $e$ at the point $x$ and add two artificial boundary nodes at this point (see Fig.~\ref{fig:network4_auxiliary-network_no_cycles}). At this additional boundary nodes we choose as boundary condition in the observer system the value of $R_\pm(t,x)$, i.e., the value of the state of the original system at the point $x$.
	If we do this for all cycles of the network, we get an auxiliary tree-shaped network, on which the difference system $\delta$ satisfies a system of equations analogous to \eqref{eq:system_diff}-\eqref{eq:diff_CC} (with adapted $\E$ and $\V$). 
	By Theorem \ref{thm:expSynchr_L2}, $\|(\delta_+, \delta_-)(t,\cdot)\|^2_{L^2(\E)}$ converges exponentially to zero for long times. This means that the state of the observer system converges exponentially towards the original solution.
\end{remark}

\begin{figure}[ht]
	\centering
	\begin{subfigure}[b]{0.49\textwidth}
		\centering
		\begin{tikzpicture}[scale=1.2]
		
		\draw[black, thick, ->] (4.5,0) -- (5.47,0.72);
		\draw[black, thick, ->] (3,0)  -- (4.45,0);
		\draw[black, thick, ->] (4.5,0) -- (5.47,-0.72);
		\draw[black, thick, ->] (5.5,0.75) -- (6.47,1.47);
		\draw[black, thick, ->] (5.5,0.75) -- (6.47,0.23);
		\draw[black, thick, ->] (5.5,0.75) -- (5.5,-0.70);
		\draw[black, thick, ->] (5.5,-0.75) -- (6.47,-1.47);

		\filldraw [black] (3,0) circle (2pt);
		\filldraw [black] (4.5,0) circle (2pt);
		\filldraw [black] (5.5,0.75) circle (2pt);
		\filldraw [black] (5.5, -0.75) circle (2pt);
		\filldraw [black] (6.5, 1.5) circle (2pt);
		\filldraw [black] (6.5, 0.2) circle (2pt);
		\filldraw [black] (6.5, -1.5) circle (2pt);
		\end{tikzpicture}
	\end{subfigure}
	\hfill
	\begin{subfigure}[b]{0.49\textwidth}
		\centering
		\begin{tikzpicture}[scale=1.2]
		
		\draw[black, thick, ->] (4.5,0) -- (5.47,0.72);
		\draw[black, thick, ->] (3,0)  -- (4.45,0);
		\draw[black, thick, ->] (4.5,0) -- (5.47,-0.72);
		\draw[black, thick, ->] (5.5,0.75) -- (6.47,1.47);
		\draw[black, thick, ->] (5.5,0.75) -- (6.47,0.23);
		\draw[black, thick, ->] (5.5,0.75) -- (5.5,0.15);
		\draw[black, thick, ->] (5.5,-0.1) -- (5.5,-0.7);
		\draw[black, thick, ->] (5.5,-0.75) -- (6.47,-1.47);

		\filldraw [black] (3,0) circle (2pt);
		\filldraw [black] (4.5,0) circle (2pt);
		\filldraw [black] (5.5,0.75) circle (2pt);
		\filldraw [black] (5.5, -0.75) circle (2pt);
		\filldraw [black] (5.5,0.1) circle (2pt);
		\filldraw [black] (5.5,-0.1) circle (2pt);
		\filldraw [black] (6.5, 1.5) circle (2pt);
		\filldraw [black] (6.5, 0.2) circle (2pt);
		\filldraw [black] (6.5, -1.5) circle (2pt);
		\end{tikzpicture}
	\end{subfigure}
		
	\caption{Construction of an auxiliary tree-shaped network (right) out of a network that contains a cycle (left) by adding one measurement point per cycle and cutting the cycle at this measurement point.}
	\label{fig:network4_auxiliary-network_no_cycles}
\end{figure}
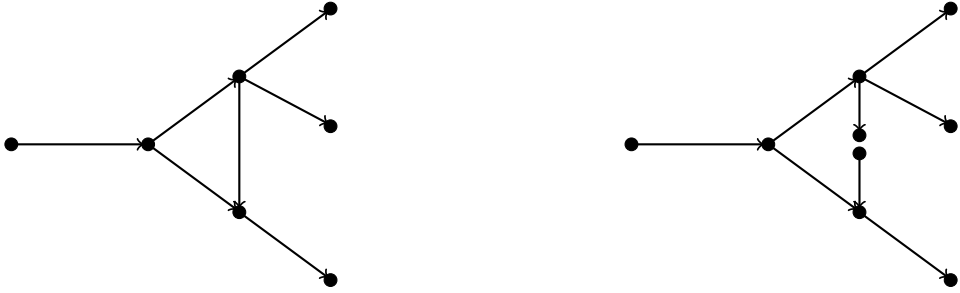

\begin{remark}
	One might wonder whether finite time synchronization holds for the system \eqref{eq:system_diff}--\eqref{eq:diff_CC} on tree-shaped networks.
	In general, the answer is no, as the following example shows.
	Consider the network shown in Fig.~\ref{fig:network_no_fin_time_synchr} together with the initial conditions
	\begin{align*}
		&\delta_\pm^3(0,x)=1, \quad x\in[0,\ell^3],\\
		&\delta_\pm^e(0,x)=0, \quad x\in[0,\ell^e], \quad  e\in\{1,2,4,5\}.
	\end{align*}
	Due to the boundary condition \eqref{eq:diff_BC} and the fact, that the value of $\delta_\pm$ is constant along the characteristic curves for $\lambda=0$, we have $\delta_+^e(t,\nu_3)=0$ for $e\in\{1,2\}$. Therefore the coupling condition \eqref{eq:diff_BC} implies
	\begin{align*}
		\delta^3_+(t,\nu_3)=-\delta^3_-(t,\nu_3)+\tfrac{2}{3} \sum_{g\in\E(\nu_3)} \delta^g\In(t,\nu_3) 
		=-\tfrac{1}{3}\delta^3_-(t,\nu_3)
	\end{align*}
	and similarly
	\begin{align*}
		\delta^3_-(t,\nu_4)=-\tfrac{1}{3}\delta^3_+(t,\nu_4).
	\end{align*}
	Using this, we can show that the system admits a solution that satisfies
	\begin{align}
		\delta^3_+(t,\nu_3)=\delta^3_-(t,\nu_4)= \left(-\tfrac{1}{3}\right)^{n+1}, \quad t\in[n \tfrac{\ell^e}{c}, (n+1)\tfrac{\ell^e}{c})
	\end{align}
	on the nodes $\nu_3$ and $\nu_4$, respectively. On edge 3, the solution can be computed from these values by
	\begin{align*}
		\delta_+^3(t,x)=\delta_+^3(t-\tfrac{x}{c}, \nu_3),\quad
		\delta_-^3(t,x)=\delta_+^3(t-\tfrac{\ell^3-x}{c}, \nu_4).
	\end{align*}
	In particular, the difference system admits a solution that converges exponentially to zero, but does not converge to zero in finite time, i.e., we have no finite time synchronization.
\end{remark}

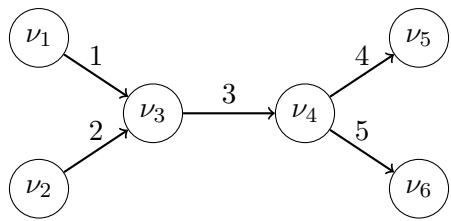
\begin{figure}[ht]
	\centering
	\begin{tikzpicture}[scale=1]
			\node (1) at (-3.5,1) [circle,draw]{$\nu_1$};
			\node (2) at (-3.5,-1) [circle,draw]{$\nu_2$};
			\node (3) at (-2,0) [circle,draw]{$\nu_3$};
			\node (4) at (0,0) [circle,draw]{$\nu_4$};
			\node (5) at (1.5,1) [circle,draw]{$\nu_5$};
			\node (6) at (1.5,-1) [circle,draw]{$\nu_6$};

			\draw[->, thick] (1) to node[above] {1} (3);
			\draw[->, thick] (2) to node[above] {2} (3);
			\draw[->, thick] (3) to node[above] {3} (4);
			\draw[->, thick] (4) to node[above] {4} (5);
			\draw[->, thick] (4) to node[above] {5} (6);
		\end{tikzpicture}
	\caption{A network with five pipes and six nodes.}
	\label{fig:network_no_fin_time_synchr}
\end{figure}

\section{Synchronization of observer for System with linear friction}   \label{sec:obs_linear_friction}
Now, we investigate the synchronization of the boundary observer for $\lambda>0$, i.e., for the case with linear source term.
Then, the difference system has the form
\begin{align}
	&\begin{pmatrix}
		\dt \delta_+^e \\ \dt \delta_-^e
	\end{pmatrix}
	+\begin{pmatrix}
		c & 0\\ 0 &-c
	\end{pmatrix}
	\begin{pmatrix}
		\dx \delta_+^e \\ \dx \delta_-^e
	\end{pmatrix}  
	= \begin{pmatrix}
		-\lambda (\delta_+^e-\delta_-^e)\\ \lambda (\delta_+^e-\delta_-^e)
	\end{pmatrix}, 
	&& e\in\E, \label{eq:system_diff_fric}\\
	&\delta^e_\pm(0,x)=y_\pm^e(x)-z_\pm^e(x)=:\tilde y_\pm^e(x),&& x\in(0,\ell^e),\  e\in\E, \\
	&\delta^e\out(t,\nu)=0,&& t\in(0,T),\  \nu\in\V_\partial,\  e\in\E(\nu), \label{eq:diff_BC_fric} \\
	&\delta^e\out(t,\nu)=-\delta^e\In(t,\nu)+\tfrac{2}{|\E(\nu)|} \sum_{g\in\E(\nu)} \delta^g\In(t,\nu),&& t\in(0,T),\  \nu\in\V\setminus\V_\partial,\  e\in\E(\nu).  \label{eq:diff_CC_fric}
\end{align}
Note that due to the friction term at the right hand side, the system state is not constant along the characteristic curves $(s, \xi_\pm(s,x,t))=(s, x\pm c(s-t))$, 
which are defined by
\begin{align*}
	\partial_s \xi_\pm(s,x,t)=\pm c, \quad \xi_\pm(t,x,t)=x.
\end{align*}

\begin{remark}  \label{rem:function_spaces}  
		The existence of solutions of hyperbolic systems of the form \eqref{eq:system_diff_fric}--\eqref{eq:diff_CC_fric} or similar can be shown by a fixed-point argument, cf. \cite{giesselmann2023observerbased,Gugat2021}.
		
		More precisely, consider the system \eqref{eq:system_diff_fric}--\eqref{eq:diff_CC_fric} and assume that the initial data are Lipschitz-continuous with Lipschitz-constant $L_I$ and are compatible with the boundary and coupling conditions. Let $T>0$ and assume that the initial data $\tilde y_\pm^e(x)$ are bounded in $C^0([0,\ell^e])$ for all $e\in\E$.
		Then, for given $S_{\max}, L_R>0$, there exists $B_{\max}>0$ and $L_I >0$ such that, if 
		$|\tilde y_\pm^e(x)| \le B_{\max} $ for all $e\in\E$, $x\in[0,\ell^e]$ and the initial data are Lipschitz-continuous with Lipschitz-constant $L_I$, then there exists a unique solution $((\delta_+^e, \delta_-^e))_{e\in\E}$ of \eqref{eq:system_diff_fric}--\eqref{eq:diff_CC_fric} in 
		\begin{align*}
			\M(S_{\max}, L_R):=\{& ((\delta_+^e, \delta_-^e))_{e\in\E} \, \vert \, (\delta_+^e, \delta_-^e)\in C^0([0,T]\times[0,\ell^e]),\\
			& |\delta_\pm^e|\le S_{\max} \text{ for all } e\in\E 
			\text{ and all $\delta_\pm^e$ are Lipschitz-continous
			}\\
			&\text{with respect to $x$ with Lipschitz-constant } L_R \}.
		\end{align*}
		This can be shown by writing the equation \eqref{eq:system_diff_fric} as an integral equation along the characteristic curves and then applying
		a fixed-point argument for the fixed-point mapping $\Phi$ that is defined by $\delta_{\pm,i+1} = \Phi_\pm(\delta_{\pm,i})$, where $\delta_{\pm,i+1} $ is the solution of
		\begin{align*}
			\delta_{\pm,i+1}^e(t,x)= \delta_{\pm,i+1}^{e}(t_\pm^e(t,x), \xi_\pm(t_\pm^e(t,x),x,t)) +\int_{t_\pm^e(t,x)}^t \mp \lambda (\delta_{+,i}^{e}-\delta_{-,i}^{e})(s, \xi_\pm(s,x,t)) ds
		\end{align*}
		for $(t,x)\in [0,T]\times[0,\ell^e]$, $e\in\E$. The existence of a unique fixed-point in $\M(S_{\max}, L_R)$ can be shown similar to \cite[Theorem 1]{Gugat2021} and \cite[Theorem 3.4]{giesselmann2023observerbased}.	
		
		Now, since $\delta_\pm^e\in \M(S_{\max}, L_R)$, we see similar to \cite[Remark 4.1]{giesselmann2023observerbased} that $\delta_\pm^e$ is 
		differentiable with respect to $x$ a.e. in $(0,\ell^e)$ and therefore also
		\begin{align*} 
			\delta_\pm^e\in C^0([0, T ]; W^{1,\infty}(0,\ell^e)) \cap W^{1,\infty}(0, T ; L^\infty(0,\ell^e)).
		\end{align*}	
\end{remark}

\begin{remark} \label{rem:existence_extension}
	In the proof of Lemma~\ref{lem:reversed_obs_inequ} we will need an extension of the solution on the pipe $[0,\ell^e]$ to a solution on the `imaginary' pipe $[\ell^e,2\ell^e]$. More precisely, for $e\in\E$, $n\in\N_0$ and $\bar t>(\tfrac{n}{2}+1)\tfrac{\ell^e}{c}$, we need the existence of a unique solution of
	\begin{align}
		&\begin{pmatrix}
			\dt \bar \delta_+^e \\ \dt \bar \delta_-^e
		\end{pmatrix}
		+\begin{pmatrix}
			c & 0\\ 0 &-c
		\end{pmatrix}
		\begin{pmatrix}
			\dx \bar \delta_+^e \\ \dx \bar \delta_-^e
		\end{pmatrix}  
		= \begin{pmatrix}
			-\lambda (\bar \delta_+^e-\bar \delta_-^e)\\ \lambda (\bar \delta_+^e-\bar \delta_-^e)
		\end{pmatrix}, \qquad 
		(t,x)\in\Omega_{\bar t, n}, \label{eq:system_extension}\\
		&\,\bar \delta^e_\pm(t,\ell^e)=\delta^e_\pm(t,\ell^e), \qquad\qquad\qquad t\in[\bar t-(\tfrac{n}{2}+1)\tfrac{\ell^e}{c},\bar t+(\tfrac{n}{2}+1)\tfrac{\ell^e}{c} ],\label{eq:BC_extension}
	\end{align}
	where $\delta^e_\pm(t,\ell^e)$ is given by the solution on $[0,\ell^e]$ and 
	\begin{align*}
		\Omega_{\bar t, n}=\{&(t,x)\in[\bar t-(\tfrac{n}{2}+1)\tfrac{\ell^e}{c},\bar t+(\tfrac{n}{2}+1)\tfrac{\ell^e}{c} ]\times [\ell^e,2\ell^e]\, \\
		&\vert\, \bar t-\tfrac{n}{2}\tfrac{\ell^e}{c}-\tfrac{2\ell^e-x}{c}\le t\le \bar t+\tfrac{n}{2}\tfrac{\ell^e}{c}+\tfrac{2\ell^e-x}{c} \},
	\end{align*}
	cf. Figure~\ref{fig:existence_sol_extension}.
	Existence of such solutions can be shown by a fixed-point argument analogously to Remark~\ref{rem:function_spaces}.
	The important point here is that every characteristic $\xi_\pm(\cdot,x,t)$, $(t,x)\in\Omega_{\bar t, n}$, hits the boundary $x=\ell^e$ at a time point $t_\pm^e(t,x)\in [\bar t-(\tfrac{n}{2}+1)\tfrac{\ell^e}{c},\bar t+(\tfrac{n}{2}+1)\tfrac{\ell^e}{c} ] $. 
\end{remark}

\begin{figure}[ht]
	\centering
	\begin{subfigure}[b]{0.49\textwidth}
		\centering
	\begin{tikzpicture}[scale=1]
		
		\draw [->,thick] (0,-2) to (0,2) node (taxis) [above] {$t$};
		\draw [->,thick] (-0.5,0) to (4.5,0) node (taxis) [right] {$x$};
		
		\foreach \x/\xtext in {0/0, 2/\ell^e, 4/2\ell^e}
		\draw[thick,shift={(\x,0)}] (0pt,2pt) -- (0pt,-2pt) node[below right] {$\xtext$};
		
		\foreach \t/\text in {-1.6/\bar t-(\tfrac{n}{2}+1)\tfrac{\ell^e}{c}\,, 1.6/ \bar t+(\tfrac{n}{2}+1)\tfrac{\ell^e}{c}\,}
		\draw[thick,shift={(0,\t)}] (-2pt, 0pt) -- (2pt,0pt) node[left] {$\text$};
		
		\draw[gray, fill, fill opacity=0.2] (2,1.6) coordinate (A) -- (4,0.8) coordinate (B)
		-- (4,-0.8) coordinate (C) --(2,-1.6) coordinate (D) -- (A);
		
		\draw[gray] (3.2,1.1) node[above] {$\Omega_{\bar t, n}$};

		\foreach \y in {1.25, 1}
		\draw[dashed, ->] (2,\y) -- (4, \y-0.8) ;
		
		\draw (2,1.2) node[left] {$\delta_-^ e$};
		
		\foreach \y in {-1.25, -1}
		\draw[dashed, ->] (2,\y) -- (4, \y+0.8) ;
		
		\draw (2,-1.2) node[left] {$\delta_+^ e$};

	\end{tikzpicture}
	\end{subfigure}
	\hfill
	\begin{subfigure}[b]{0.49\textwidth}
		\centering
		\begin{tikzpicture}[scale=1]
				
			\draw [->,thick] (0,-2) to (0,2) node (taxis) [above] {$t$};
			\draw [->,thick] (-0.5,0) to (3,0) node (taxis) [right] {$x$};
			
			\foreach \x/\xtext in {0/0, 2/\ell^e}
			\draw[thick,shift={(\x,0)}] (0pt,2pt) -- (0pt,-2pt) node[below right] {$\xtext$};
			
			\foreach \t/\text in {-1.6/\bar t-\tfrac{\ell^e}{c}\,, 1.6/ \bar t+\tfrac{\ell^e}{c}\,}
			\draw[thick,shift={(0,\t)}] (-2pt, 0pt) -- (2pt,0pt) node[left] {$\text$};
			
			\draw[gray, fill, fill opacity=0.2] (0,1.6) coordinate (A) -- (2,0) coordinate (B)
			-- (2,0) coordinate (C) --(0,-1.6) coordinate (D) -- (A);
			
			\draw[gray] (1.2,1) node[above] {$\Omega_{t}$};

		\end{tikzpicture}
	\end{subfigure}
	
	\caption{Extension of the solution of \eqref{eq:system_diff_fric}--\eqref{eq:diff_CC_fric} on the pipe $[0,\ell^e]$ to a solution of \eqref{eq:system_extension}--\eqref{eq:BC_extension} on the pipe $[\ell^e,2\ell^e]$ (left) and sketch of $\Omega_t$ from Lemma~\ref{lem:obs_inequ_single_pipe}.}
	\label{fig:existence_sol_extension}
\end{figure}
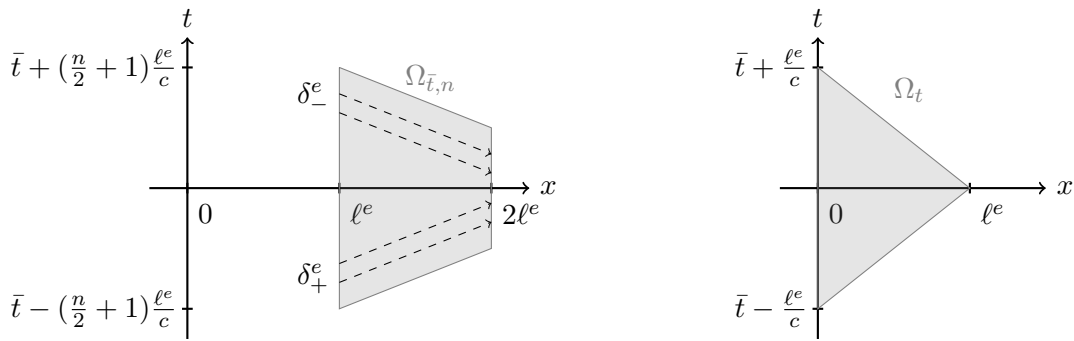

As for the system without friction, we want to identify the minimal number of measurements that are needed in addition to boundary measurements in order to guarantee synchronization of the state of the observer system towards the original system state.
Showing this synchronization is equivalent to showing the convergence of the difference system \eqref{eq:system_diff_fric}--\eqref{eq:diff_CC_fric} to zero in the long time limit.
First, we present an example that shows that, if only boundary measurements are given, in general we cannot expect convergence for networks that contain a cycle, at least no convergence that is faster than the decay due to the friction term.
Second, we show that for networks that do not contain a cycle, i.e., for tree-shaped networks, the state of the difference system converges exponentially to zero with a rate that is uniform in the $\lambda \rightarrow 0$ limit.
As before, for networks that contain cycles we can add one measurement point in every cycle in order to guarantee synchronization, see Remark~\ref{rem:synchr_not_tree_shaped}.

\begin{example}
	We consider a general network with a cycle that does not contain a boundary node, i.e., we consider the same network as in Section~\ref{sec:example_no_synchr} and complement it with the same initial conditions \eqref{eq:exl_no_synchr-IC1}--\eqref{eq:exl_no_synchr-IC2}.
	As shown in Section~\ref{sec:example_no_synchr}, the initial conditions are compatible with the boundary and coupling conditions.
	Since now we have a nonzero source term, this no longer implies that the solution is constant over time, but one can show that
	\begin{align}
		&\delta_+^e(t,x)=a\exp(-2\lambda t), \quad \delta_+^e(t,x)=-a\exp(-2\lambda t),\quad e\in\E_c,\ t>0,\ x\in(0,\ell^e) \label{eq:example_cycle_sol1}\\
		&\delta_\pm^e(t,x)=0, \quad e\in\E_n,\ t>0,\ x\in(0,\ell^e) \label{eq:example_cycle_sol2}
	\end{align} 
	is a solution of \eqref{eq:system_diff_fric}--\eqref{eq:diff_CC_fric} with initial conditions \eqref{eq:exl_no_synchr-IC1}--\eqref{eq:exl_no_synchr-IC2}.
	Note that by Remark~\ref{rem:function_spaces} 
	there exists a unique solution of \eqref{eq:system_diff_fric}--\eqref{eq:diff_CC_fric}, i.e., \eqref{eq:example_cycle_sol1}--\eqref{eq:example_cycle_sol2} is the unique solution.
	This solution is constant in space on each pipe and decays only with a rate that results from the friction parameter. In particular, we have no additional decay resulting from the boundary condition \eqref{eq:diff_BC_fric}.
\end{example}

As a first step towards the exponential synchronization of the observer system for tree-shaped networks, we show a $L^2$-observability inequality for a single pipe. The proof is similar to the proof of \cite[Theorem 2]{Gugat2021}.

\begin{lem}\label{lem:obs_inequ_single_pipe}
	Let $e\in\E$, $t>\tfrac{\ell^e}{c}$ and let $(\delta_+^e, \delta_-^e)\in W^{1,\infty}(\Omega_t)$  
	be a solution of \eqref{eq:system_diff_fric} on
	\begin{align*}
		\Omega_t:=\{(s,x)\in[t-\tfrac{\ell^e}{c},t+\tfrac{\ell^e}{c} ]\times [0,\ell^e] \, \vert\, t-\tfrac{\ell^e-x}{c}\le s\le t+\tfrac{\ell^e-x}{c} \}
	\end{align*}
	(see Figure~\ref{fig:existence_sol_extension}). 
	Then there exists a constant $C>0$ such that
	\begin{align*}
		\int_0^{\ell^e} \left( |\delta_+^e(t,x)|^2+  |\delta_-^e(t,x)|^2\right) dx
		\le C \int_{t-\tfrac{\ell^e}{c}}^{t+\tfrac{\ell^e}{c}} \left( |\delta_+^e(s,0)|^2+  |\delta_-^e(s,0)|^2\right)  ds.
	\end{align*}
\end{lem}

\begin{proof}
	For simplicity of notation we leave out the superscript $e$. 
	By multiplying equation \eqref{eq:system_diff_fric} with $\delta_+$ we have
	\begin{align}
		&|\delta_+|^2(t,x)-|\delta_+|^2(t-\tfrac{x}{c},0)=
		\int_{t-\tfrac{x}{c}}^{t} \ds |\delta_+|^2(s, x+c(s-t)) \, ds  \notag \\
		&= \int_{t-\tfrac{x}{c}}^{t} -2\lambda\,\delta_+ (\delta_+-\delta_-)(s, x+c(s-t)) \, ds
		\le \tfrac{\lambda}{2} \int_{t-\tfrac{x}{c}}^{t} \left(|\delta_+|^2 +|\delta_-|^2 \right) (s, x+c(s-t)) \, ds. \label{eq:estimate_deltap_obs_ineq}
	\end{align}
	This implies
	\begin{align*}
		\int_0^\ell |\delta_+(t,x)|^2 dx
		\le&\,  \,\int_0^\ell|\delta_+(t-\tfrac{x}{c},0)|^2 dx
		+ \tfrac{\lambda}{2}  \int_0^\ell \int_{t-\tfrac{x}{c}}^{t} \left(|\delta_+|^2+|\delta_-|^2\right)(s, x+c(s-t)) \, ds\, dx\\
		=&\, c \int_{t-\tfrac{\ell}{c}}^{t} |\delta_+(s,0)|^2 ds
		+ \tfrac{\lambda}{2}  \int_0^\ell \int_{t-\tfrac{\ell-x}{c}}^{t} \left(|\delta_+|^2+|\delta_-|^2\right)(s, x) \, ds\, dx,
	\end{align*}
	where in the second step we first changed the integration order of the second integral and then substituted $x+c(s-t)$ by $x$.
	Similarly, we can show
	\begin{align*}
		\int_0^\ell |\delta_-(t,x)|^2 dx
		\le c \int_{t}^{t+\tfrac{\ell}{c}} |\delta_-(s,0)|^2 ds
		+ \tfrac{5}{2} \lambda  \int_0^\ell \int_{t}^{t+\tfrac{\ell-x}{c}} \left(|\delta_+|^2+|\delta_-|^2\right)(s, x) \, ds\, dx.
	\end{align*}
	For $\delta_-$ we have a larger constant in front of the second integral since in this case the signs do not allow the same cancellation as in \eqref{eq:estimate_deltap_obs_ineq}. 
	
	Now, we define 
	\begin{align*}
		\H(x):= \int_{t-\tfrac{\ell-x}{c}}^{t+\tfrac{\ell-x}{c}} \left( |\delta_+(s,x)|^2+  |\delta_-(s,x)|^2\right)  ds.
	\end{align*}
	Then
	\begin{align*}
		\tfrac{d}{dx} \H(x)
		=&-\tfrac{1}{c} \left(|\delta_+|^2+|\delta_-|^2 \right)(t+\tfrac{\ell-x}{c})
			-\tfrac{1}{c} \left(|\delta_+|^2+|\delta_-|^2 \right)(t-\tfrac{\ell-x}{c})\\
			&+\int_{t-\tfrac{\ell-x}{c}}^{t+\tfrac{\ell-x}{c}} \dx \left( |\delta_+(s,x)|^2+  |\delta_-(s,x)|^2\right)  ds\\
		=&-\tfrac{1}{c} \left(|\delta_+|^2+|\delta_-|^2 \right)(t+\tfrac{\ell-x}{c})
			-\tfrac{1}{c} \left(|\delta_+|^2+|\delta_-|^2 	\right)(t-\tfrac{\ell-x}{c})\\
			&+\int_{t-\tfrac{\ell-x}{c}}^{t+\tfrac{\ell-x}{c}}  \left( 	-\tfrac{1}{c}\dt |\delta_+|^2+  \tfrac{1}{c} \dt |\delta_-|^2 -2 \tfrac{\lambda}{c} (\delta_+^2-\delta_-^2)\right)(s,x)  ds\\
		=& -\tfrac{2}{c} |\delta_+|^2(t+\tfrac{\ell-x}{c})-\tfrac{2}{c} 	|\delta_-|^2(t-\tfrac{\ell-x}{c})
			-2 \tfrac{\lambda}{c}\int_{t-\tfrac{\ell-x}{c}}^{t+\tfrac{\ell-x}{c}}    (\delta_+^2-\delta_-^2)(s,x)  ds\\
		\le & \,2\,\tfrac{\lambda}{c}\, \H(x).
	\end{align*}
	Using a Gronwall lemma this implies
	\begin{align*}
		\H(x)\le \exp(2\,\tfrac{\lambda}{c}\,\ell)\, \H(0).
	\end{align*}

	In summary, we have shown
	\begin{align*}
		&\int_0^\ell \left(|\delta_+(t,x)|^2+ |\delta_-(t,x)|^2\right) dx
		\le c \int_{t-\tfrac{\ell}{c}}^{t+\tfrac{\ell}{c}} \left(|\delta_+(s,0)|^2 + |\delta_-(s,0)|^2\right) ds
		+ \tfrac{5}{2} \lambda 
		\int_0^\ell \H(x) dx\\
		&\qquad\le \left(  c + \tfrac{5}{2} \lambda 
		\exp(2\,\tfrac{\lambda}{c}\,\ell) \right) \int_{t-\tfrac{\ell}{c}}^{t+\tfrac{\ell}{c}} \left(|\delta_+(s,0)|^2 + |\delta_-(s,0)|^2\right) ds.   \qedhere
	\end{align*}
\end{proof}

\begin{remark}
	In the limit $\lambda\rightarrow 0$, i.e. in the no-friction limit, the constant in the observability inequality im Lemma~\ref{lem:obs_inequ_single_pipe} converges to $c>0$. In particular, the observability constant does not go to infinity and is bounded for $\lambda\rightarrow 0$.
\end{remark}

In order to show an $L^2$-observability inequality for the system with linear friction on a tree-shaped network, we have to estimate the value of $\delta_\pm$ at one end of a pipe by the value of $\delta_\pm$ at the other end of the pipe. For this, we will use the following lemma.
\begin{lem}\label{lem:reversed_obs_inequ}
	Let $(\delta_+^e, \delta_-^e), e\in\E$ be a solution of \eqref{eq:system_diff_fric}--\eqref{eq:diff_CC_fric} on a network $G=(\V, \E)$ and let $t>\tfrac{5}{2} \tfrac{\lmax}{c}$. Then we can estimate
	\begin{align} \label{eq:inv_obs_inequality}
		\int_{t-\tfrac{1}{2}\tfrac{\ell^e}{c}}^{t+\tfrac{1}{2}\tfrac{\ell^e}{c}}
		\left(|\delta^e_+(s,\ell^e) |^2+ |\delta^e_-(s,\ell^e) |^2\right) ds
		\le C_2 \int_{t-\tfrac{5}{2}\tfrac{\ell^e}{c}}^{t+\tfrac{5}{2}\tfrac{\ell^e}{c}}
		\left(|\delta^e_+(s,0) |^2+ |\delta^e_-(s,0) |^2\right) ds
	\end{align}
	for all $e\in\E$ with a constant $C_2>0$.
\end{lem}

\begin{proof}
	The main strategy of the proof is to estimate the time integral on the left hand side of \eqref{eq:inv_obs_inequality} by a space integral and then apply the observability inequality for a single pipe, see Lemma~\ref{lem:obs_inequ_single_pipe}.
	
	Let $e\in\E$. We start by noting that \eqref{eq:system_diff_fric} implies
	\begin{align*}
		&|\delta^e_+|^2(s,\ell^e)-|\delta^e_+|^2(t-\tfrac{1}{2} 	\tfrac{\ell^e}{c},\tfrac{1}{2}\ell^e-c(s-t))
		=\int_{t-\tfrac{1}{2} \tfrac{\ell^e}{c}}^{s} \tfrac{d}{d \tilde s}
		\left(|\delta^e_+|^2(\tilde s, \ell^e+c(\tilde s - s)) \right)  d\tilde s\\
		&= \int_{t-\tfrac{1}{2} \tfrac{\ell^e}{c}}^{s}
		(-2\lambda)  \delta_+^e(\delta_+^e-\delta_-^e)(\tilde s, \ell^e+c(\tilde s - s)) d\tilde s\\
		&\le \tfrac{\lambda}{2} \int_{t-\tfrac{1}{2} \tfrac{\ell^e}{c}}^{s}
		(|\delta_+^e|^2+|\delta_-^e|^2)(\tilde s, \ell^e+c(\tilde s - s)) d\tilde s.
	\end{align*}
	Integrating in time yields 
	\begin{align*}
		&\int_{t-\tfrac{1}{2}\tfrac{\ell^e}{c}}^{t+\tfrac{1}{2}\tfrac{\ell^e}{c}}
		|\delta^e_+(s,\ell^e) |^2 ds
		\le \int_{t-\tfrac{1}{2}\tfrac{\ell^e}{c}}^{t+\tfrac{1}{2}\tfrac{\ell^e}{c}} |\delta^e_+(t-\tfrac{1}{2} \tfrac{\ell^e}{c},\tfrac{1}{2}\ell^e-c(s-t))|^2 ds\\
		&\quad + \tfrac{\lambda}{2}
		\int_{t-\tfrac{1}{2}\tfrac{\ell^e}{c}}^{t+\tfrac{1}{2}\tfrac{\ell^e}{c}} \int_{t-\tfrac{1}{2} \tfrac{\ell^e}{c}}^{s}
		\left(|\delta_+^e|^2+|\delta_-^e|^2\right)(\tilde s, \ell^e+c(\tilde s - s)) d\tilde s\, ds\\
		& = \tfrac{1}{c} \int_0^{\ell^e} |\delta^e_+(t-\tfrac{1}{2} \tfrac{\ell^e}{c},x)|^2 dx\,
		+ \tfrac{\lambda}{2} 
		\int_{t-\tfrac{1}{2}\tfrac{\ell^e}{c}}^{t+\tfrac{1}{2}\tfrac{\ell^e}{c}} \int_{\tilde s}^{t+\tfrac{1}{2} \tfrac{\ell^e}{c}}
			\left(|\delta_+^e|^2+|\delta_-^e|^2\right)(\tilde s, \ell^e+c(\tilde s - s)) ds\, d\tilde s\\
		&= \tfrac{1}{c} \int_0^{\ell^e} |\delta^e_+(t-\tfrac{1}{2} \tfrac{\ell^e}{c},x)|^2 dx\,
		+ \tfrac{\lambda}{2 c} 
		\int_{t-\tfrac{1}{2}\tfrac{\ell^e}{c}}^{t+\tfrac{1}{2}\tfrac{\ell^e}{c}} \int_{\tfrac{1}{2}\ell^e +c(\tilde s-t)}^{\ell^e}
		\left(|\delta_+^e|^2+|\delta_-^e|^2\right)(\tilde s, x) dx\, d\tilde s.
	\end{align*}
	Now, we denote the inner integral of the second term by 
	\begin{align*}
		\tilde\G_+^e(\tilde s):=\int_{\tfrac{1}{2}\ell^e +c(\tilde s-t)}^{\ell^e}
		\left(|\delta_+^e|^2+|\delta_-^e|^2\right)(\tilde s, x) dx.
	\end{align*}
	This is an integral of $|\delta_+^e|^2+|\delta_-^e|^2$ along a line from $\tfrac{1}{2}\ell^e +c(\tilde s-t)$ to $\ell^e$, see Figure~\ref{fig:def_G+}.
	In order to estimate the time derivative of $\tilde\G_+^e$, we extend the integral to an integral along the line from $\tfrac{1}{2}\ell^e +c(\tilde s-t)$ to $\tfrac{3}{2}\ell^e -c(\tilde s-t)$, i.e., 
	\begin{align*}
		\G_+^e(\tilde s):=
		\int_{\tfrac{1}{2}\ell^e +c(\tilde s-t)}^{\tfrac{3}{2}\ell^e -c(\tilde s-t)}
		\left(|\delta_+^e|^2+|\delta_-^e|^2\right)(\tilde s, x) dx
	\end{align*}
	with $\tilde\G_+^e(\tilde s)\le \G_+^e(\tilde s)$ 
	for $\tilde s\in[t-\tfrac{1}{2}\tfrac{\ell^e}{c}, t+\tfrac{1}{2}\tfrac{\ell^e}{c}]$.
	For the definition of $\G_+^e$ we have to extend the solution $\delta_\pm^e$ on $[0,\ell^e]$ to a solution on the `imaginary' pipe $[\ell^e,2 \ell^e]$. 
	We do this by extending the values of $\delta_\pm^e(\tilde s, \ell^e)$, $\tilde s \in[t-\tfrac{3}{2}\tfrac{\ell^e}{c}, t+\tfrac{1}{2}\tfrac{\ell^e}{c}]$, along the characteristic curves to a solution in the area
	\begin{align} \label{eq:extension_area}
		\tilde s \in[t-\tfrac{3}{2}\tfrac{\ell^e}{c}, t+\tfrac{1}{2}\tfrac{\ell^e}{c}], \quad
		\ell^e\le x \le \min\{ \tfrac{3}{2}\ell^e -c(\tilde s-t), \tfrac{5}{2}\ell^e +c(\tilde s-t)  \},
	\end{align}
	see Figure~\ref{fig:G+_sol_extension} as well as Remark~\ref{rem:existence_extension} and Figure~\ref{fig:existence_sol_extension}.

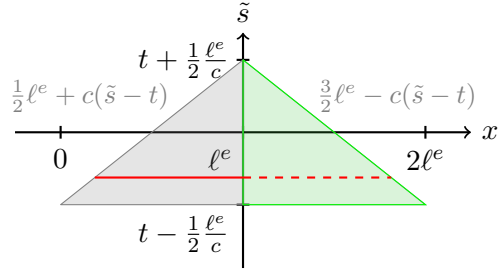
\begin{figure}[h]
	\centering
	
	\begin{tikzpicture}[scale=1.2]
		
		\draw [->,thick] (2,-1.5) to (2,1.1) node (taxis) [above] {$\tilde s$};
		\draw [->,thick] (-0.5,0) to (4.5,0) node (taxis) [right] {$x$};
		
		\foreach \x/\xtext in {0/0, 4/2\ell^e}
		\draw[thick,shift={(\x,0)}] (0pt,2pt) -- (0pt,-2pt) node[below] {$\xtext$};
		
		\draw[thick,shift={(2,0)}] (0pt,2pt) -- (0pt,-2pt) node[below left] {$\ell^e$};
		
		\draw[thick,shift={(2,-0.8)}] (-2pt, 0pt) -- (2pt,0pt) node[below left] {$ t-\tfrac{1}{2}\tfrac{\ell^e}{c}\ $};

		\draw[thick,shift={(2,0.8)}] (-2pt, 0pt) -- (2pt,0pt) node[left] {$ t+\tfrac{1}{2}\tfrac{\ell^e}{c}\ $};

		\draw[gray, fill, fill opacity=0.2] (0,-0.8) coordinate (A) -- (2,0.8) coordinate (B)
		-- (2,-0.8) coordinate (C) -- (A);
		
		\draw[gray, fill, fill opacity=0.1] (B) -- (4,-0.8) coordinate (D)
		-- (C) -- (B);
		
		\draw[green, fill, fill opacity=0.1.5]  (B) --  (C) -- (D) -- (B);
		
		\draw[thick, red] (3/8, -0.5)--(2,-0.5);
		\draw[thick, dashed, red] (2,-0.5) -- (3+5/8, -0.5);
		
		\draw[gray] (1.3,0.4) node[left] {{\small $\tfrac{1}{2}\ell^e +c(\tilde s-t)$}};
		\draw[gray] (2.7,0.4) node[right] {{\small $\tfrac{3}{2}\ell^e -c(\tilde s-t)$}};
		
	\end{tikzpicture}
	
	\caption{Illustration of the definition of $\tilde \G_+(\cdot) $ (integral along the red solid line) and of $\G_+(\cdot) $ (integral along the red solid and dashed line).
	}
	\label{fig:def_G+}
\end{figure}

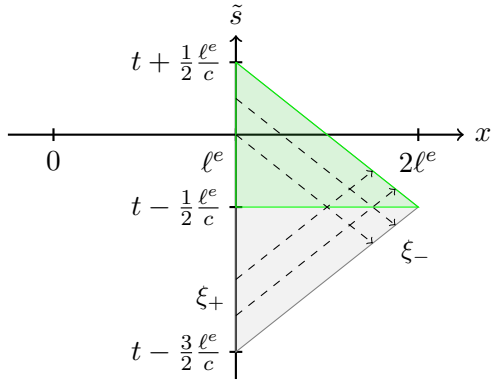
\begin{figure}[ht]
	\centering
	
	\begin{tikzpicture}[scale=1.2]
			
		\draw [->,thick] (2,-2.7) to (2,1.1) node (taxis) [above] {$\tilde s$};
		\draw [->,thick] (-0.5,0) to (4.5,0) node (taxis) [right] {$x$};
		
		\foreach \x/\xtext in {0/0, 4/2\ell^e}
		\draw[thick,shift={(\x,0)}] (0pt,2pt) -- (0pt,-2pt) node[below] {$\xtext$};
		
		\draw[thick,shift={(2,0)}] (0pt,2pt) -- (0pt,-2pt) node[below left] {$\ell^e$};
		
		\draw[thick,shift={(2,-0.8)}] (-2pt, 0pt) -- (2pt,0pt) node[left] {$ t-\tfrac{1}{2}\tfrac{\ell^e}{c}\ $};
		
		\draw[thick,shift={(2,-2.4)}] (-2pt, 0pt) -- (2pt,0pt) node[left] {$ t-\tfrac{3}{2}\tfrac{\ell^e}{c}\ $};
		
		\draw[thick,shift={(2,0.8)}] (-2pt, 0pt) -- (2pt,0pt) node[left] {$ t+\tfrac{1}{2}\tfrac{\ell^e}{c}\ $};

		\draw[gray, fill, fill opacity=0.1] (2,0.8) coordinate (A) -- (4,-0.8) coordinate (B)
		-- (2,-2.4) coordinate (C) -- (A);
		
		\draw[green, fill, fill opacity=0.1.5]  (A) --  (B)
		-- (2,-0.8) coordinate (D) -- (A);
		
		\draw[dashed, ->] (2,0.4) -- (3.75, -1) ;
		\draw[dashed, ->] (2,0) -- (3.5, -1.2) ;
		\draw (3.7, -1.3) node[right] {{\small $\xi_-$}};
		
		\draw[dashed, ->] (2,-2) -- (3.75, -0.6) ;
		\draw[dashed, ->] (2,-1.6) -- (3.5, -0.4) ;
		\draw (2, -1.8) node[left] {{\small $\xi_+$}};
				
	\end{tikzpicture}
	
	\caption{Extension of the solution $\delta_\pm^e$ to a solution on the imaginary pipe $[\ell^e, 2\ell^e]$ that is defined for $(\tilde s, x)$ satisfying \eqref{eq:extension_area}, i.e. for $(\tilde s, x)$ in the shaded area. The characteristics $\xi_\pm$ are drawn as dashed lines.
		The green region is the same as in Figure~\ref{fig:def_G+}.
	}
	\label{fig:G+_sol_extension}
\end{figure}

	Then, the time derivative of $\G_+^e(\tilde s)$ can be estimated by
	\begin{align*}
		\tfrac{\partial \G_+^e(\tilde s)}{\partial \tilde s}
		=& -c \left(|\delta_+^e|^2+|\delta_-^e|^2\right)(\tilde s, \tfrac{3}{2}\ell^e -c(\tilde s-t))
			-c \left(|\delta_+^e|^2+|\delta_-^e|^2\right)(\tilde s, \tfrac{1}{2}\ell^e +c(\tilde s-t))\\
			&\ + \int_{\tfrac{1}{2}\ell^e +c(\tilde s-t)}^{\tfrac{3}{2}\ell^e -c(\tilde s-t)} \tfrac{\partial}{\partial \tilde s} 	\left(|\delta_+^e|^2+|\delta_-^e|^2\right)(\tilde s, x) dx\\
		=& -c \left(|\delta_+^e|^2+|\delta_-^e|^2\right)(\tilde s, \tfrac{3}{2}\ell^e -c(\tilde s-t))
			-c \left(|\delta_+^e|^2+|\delta_-^e|^2\right)(\tilde s, 	\tfrac{1}{2}\ell^e +c(\tilde s-t))\\
			&\ + \int_{\tfrac{1}{2}\ell^e +c(\tilde s-t)}^{\tfrac{3}{2}\ell^e 	-c(\tilde s-t)} 	\left( -c \dx |\delta_+^e|^2+c \dx  |\delta_-^e|^2 -2\lambda |\delta_+^e-\delta_-^e|^2\right)(\tilde s, x) dx\\
		=& -2c|\delta_+^e|^2(\tilde s, \tfrac{3}{2}\ell^e -c(\tilde s-t))
			-2c|\delta_-^e|^2(\tilde s, \tfrac{1}{2}\ell^e +c(\tilde s-t))\\
			&-2 \lambda \int_{\tfrac{1}{2}\ell^e +c(\tilde s-t)}^{\tfrac{3}{2}\ell^e -c(\tilde s-t)} |\delta_+^e-\delta_-^e|^2(\tilde s, x) dx.
	\end{align*}
	In particular, we have
	\begin{align*}
		\tfrac{\partial \G_+^e(\tilde s)}{\partial \tilde s}
		\le 0,
	\end{align*}
	i.e., 
	\begin{align*}
		\tilde \G_+^e(\tilde s)\le \G_+^e(\tilde s)
		\le \G_+^e(t-\tfrac{1}{2}\tfrac{\ell^e}{c})
		= \int_{0}^{2\ell^e}
		\left(|\delta_+^e|^2+|\delta_-^e|^2\right)(t-\tfrac{1}{2}\tfrac{\ell^e}{c}, x) dx
	\end{align*}
	for all $\tilde s\in[t-\tfrac{1}{2}\tfrac{\ell^e}{c}, t+\tfrac{1}{2}\tfrac{\ell^e}{c}]$.
	Note that the proof of Lemma~\ref{lem:obs_inequ_single_pipe} only uses equation \eqref{eq:system_diff_fric} as well as the existence of solutions. Therefore we can apply the observability inequality in Lemma~\ref{lem:obs_inequ_single_pipe} also to the interval $[0,2 \ell^e]$, which implies
	\begin{align*}
		\tilde \G_+^e(\tilde s)
		\le C \int_{t-\tfrac{5}{2}\tfrac{\ell^e}{c}}^{t+\tfrac{3}{2}\tfrac{\ell^e}{c}}
		\left(|\delta_+^e|^2+|\delta_-^e|^2\right)(s,0) ds.
	\end{align*}	
	Thus, we can summarize the estimate for $\delta^e_+$ as
	\begin{align*}
		\int_{t-\tfrac{1}{2}\tfrac{\ell^e}{c}}^{t+\tfrac{1}{2}\tfrac{\ell^e}{c}}
		|\delta^e_+(s,\ell^e) |^2 ds
		&\le\tfrac{1}{c} \int_0^{\ell^e} |\delta^e_+(t-\tfrac{1}{2} \tfrac{\ell^e}{c},x)|^2 dx\,
		+ \tfrac{\lambda}{2c} 
		\int_{t-\tfrac{1}{2}\tfrac{\ell^e}{c}}^{t+\tfrac{1}{2}\tfrac{\ell^e}{c}} 
		\tilde \G_+^e(\tilde s) d\tilde s\\
		&\le C \left( \tfrac{1}{c}+ \tfrac{ \lambda}{2c} \tfrac{\ell^e}{c}\right)
		\int_{t-\tfrac{5}{2}\tfrac{\ell^e}{c}}^{t+\tfrac{3}{2}\tfrac{\ell^e}{c}}
		\left(|\delta_+^e|^2+|\delta_-^e|^2\right)(s,0) ds. 
	\end{align*}
	
	It remains to find a bound for the time integral of $|\delta_-^e(s,\ell^e)|^2$. Similarly as for $\delta_+^e$ we can estimate
	\begin{multline*}
		\quad|\delta^e_-(s,\ell^e) |^2
		\le   |\delta^e_-(t+\tfrac{1}{2} \tfrac{\ell^e}{c},\tfrac{1}{2}\ell^e+c(s-t))|^2\\
		\quad + \tfrac{5}{2}\lambda 
		\int_s^{t+\tfrac{1}{2} \tfrac{\ell^e}{c}}
		\left(|\delta_+^e|^2+|\delta_-^e|^2\right)(\tilde s, \ell^e-c(\tilde s - s)) d\tilde s,\quad
	\end{multline*}
	and therefore
	\begin{multline*}
		\quad\int_{t-\tfrac{1}{2}\tfrac{\ell^e}{c}}^{t+\tfrac{1}{2}\tfrac{\ell^e}{c}}
		|\delta^e_-(s,\ell^e) |^2 ds
		\le   \tfrac{1}{c} \int_0^{\ell^e} |\delta^e_-(t+\tfrac{1}{2} \tfrac{\ell^e}{c},x)|^2 dx \\
		+ \tfrac{5}{2} \tfrac{\lambda}{c} 
		\int_{t-\tfrac{1}{2}\tfrac{\ell^e}{c}}^{t+\tfrac{1}{2}\tfrac{\ell^e}{c}} \int_{\tfrac{1}{2}\ell^e -c(\tilde s-t)}^{\ell^e}
		\left(|\delta_+^e|^2+|\delta_-^e|^2\right)(\tilde s, x) dx\, d\tilde s.\quad
	\end{multline*}
	Again, we extend the solution to $[0,2\ell^e]$ and define
	\begin{align*}
		\tilde \G_-^e(\tilde s) := \int_{\tfrac{1}{2}\ell^e -c(\tilde s-t)}^{\ell^e}
		\left(|\delta_+^e|^2+|\delta_-^e|^2\right)(\tilde s, x) dx
		\le  \int_{\tfrac{1}{2}\ell^e -c(\tilde s-t)}^{\tfrac{3}{2}\ell^e +c(\tilde s-t)}
		\left(|\delta_+^e|^2+|\delta_-^e|^2\right)(\tilde s, x) dx
		=:\G_-^e(\tilde s) 
	\end{align*}
	for $\tilde s\in[t-\tfrac{1}{2}\tfrac{\ell^e}{c}, t+\tfrac{1}{2}\tfrac{\ell^e}{c}]$.
	Estimating the time derivative of $\G_-^e(\tilde s)$ yields
	\begin{align*}
		\tfrac{\partial \G_-^e(\tilde s)}{\partial \tilde s}
		=&\, c \left(|\delta_+^e|^2+|\delta_-^e|^2\right)(\tilde s, \tfrac{3}{2}\ell^e +c(\tilde s-t))
		+c \left(|\delta_+^e|^2+|\delta_-^e|^2\right)(\tilde s, 	\tfrac{1}{2}\ell^e -c(\tilde s-t))\\
		&\ + \int_{\tfrac{1}{2}\ell^e -c(\tilde s-t)}^{\tfrac{3}{2}\ell^e 	+c(\tilde s-t)} 	\left( -c \dx |\delta_+^e|^2+c \dx  |\delta_-^e|^2 -2\lambda |\delta_+^e-\delta_-^e|^2\right)(\tilde s, x) dx\\
		=& \,2c\, |\delta_-^e|^2(\tilde s, \tfrac{3}{2}\ell^e +c(\tilde s-t))
		+2c\, |\delta_+^e|^2(\tilde s, \tfrac{1}{2}\ell^e -c(\tilde s-t))\\
		&-2 \lambda \int_{\tfrac{1}{2}\ell^e +c(\tilde s-t)}^{\tfrac{3}{2}\ell^e -c(\tilde s-t)} |\delta_+^e-\delta_-^e|^2(\tilde s, x) dx\\
		\ge &\, -4\lambda \G_-^e(\tilde s).
	\end{align*}
	By applying a Gronwall lemma we see
	\begin{align*}
		\G_-^e(\tilde s) \le \exp(4\lambda \tfrac{\ell^e}{c}) \, \G_-^e(t+\tfrac{1}{2}\tfrac{\ell^e}{c}) 
		= \exp(4\lambda \tfrac{\ell^e}{c}) \,
		\int_{0}^{2\ell^e}
		\left(|\delta_+^e|^2+|\delta_-^e|^2\right)(t+\tfrac{1}{2}\tfrac{\ell^e}{c}, x) dx,
	\end{align*}
	where the right hand side can be estimated by the observability inequality Lemma~\ref{lem:obs_inequ_single_pipe}.

	In summary, we have shown
	\begin{align*}
		&\int_{t-\tfrac{1}{2}\tfrac{\ell^e}{c}}^{t+\tfrac{1}{2}\tfrac{\ell^e}{c}}
		\left( |\delta^e_+(s,\ell^e) |^2 + |\delta^e_-(s,\ell^e) |^2\right) ds\\
		&\le \tfrac{1}{c} \int_0^{\ell^e} \left( |\delta^e_+(t-\tfrac{1}{2} \tfrac{\ell^e}{c},x)|^2 + |\delta^e_-(t+\tfrac{1}{2} \tfrac{\ell^e}{c},x)|^2\right) dx\,
		+ \tfrac{5}{2}\tfrac{\lambda}{c}
		\int_{t-\tfrac{1}{2}\tfrac{\ell^e}{c}}^{t+\tfrac{1}{2}\tfrac{\ell^e}{c}} \left(\tilde\G_+^e(\tilde s)+\tilde\G_-^e(\tilde s)\right) d\tilde s\\
		&\le C_2 
		\int_{t-\tfrac{5}{2}\tfrac{\ell^e}{c}}^{t+\tfrac{5}{2}\tfrac{\ell^e}{c}}
		\left(|\delta_+^e(s,0)|^2+|\delta_-^e(s,0)|^2\right) ds 
	\end{align*}
	with $C_2=2C \left(\tfrac{1}{c}+ \tfrac{5}{2}\tfrac{\lambda}{c}\tfrac{\ell^e}{c} \exp(4\lambda \tfrac{\ell^e}{c}) \right) $, where C is the constant from Lemma~\ref{lem:obs_inequ_single_pipe}.
\end{proof}

\begin{remark}
	Again, the constant in \eqref{eq:inv_obs_inequality} is bounded in the limit $\lambda\rightarrow 0$. More precisely, we have $C_2\rightarrow 2$ for $\lambda\rightarrow 0$.
\end{remark}

Now we can show a $L^2$-observability inequality for tree-shaped networks. This will be the main ingredient in the proof of exponential synchronization.

\begin{thm}[$L^2$-observability inequality]  \label{thm:ObsInequ_L2_fric}
	Let $G=(\V, \E)$ be a tree-shaped network with $N$ inner nodes.
	Define the sequence $(b_n)_{n\in\N}$ by $b_n:=2n-1$.
	Then there exists a constant $C_3>0$ such that for $T\ge b_N \tfrac{\lmax}{c}$ and $t>T$ we have
	\begin{align} \label{eq:obsInequ_L2_fric}
		\|(\delta_+, \delta_-)(t,\cdot)\|^2_{L^2(\E)}
		\le C_3 \sum_{\nu\in\V_\partial}  \|(\delta_+, \delta_-)(\cdot, \nu)\|^2_{L^2([t-T, t+T])},
	\end{align}
	where
	\begin{align*}
		\|(\delta_+, \delta_-)(t,\cdot)\|^2_{L^2(\E)}
		:=\|\delta_+(t,\cdot)\|^2_{L^2(\E)} + \|\delta_-(t,\cdot)\|^2_{L^2(\E)}.
	\end{align*}
\end{thm}

\begin{proof}
	We use the same notation as in the proof of Lemma~\ref{lem:ObsInequ_L2}. Again, we use an induction argument over the number of inner nodes of the network.
	Assume that the observability inequality \eqref{eq:obsInequ_L2_fric} holds for all tree-shaped networks with at most $N-1$ inner nodes (we assume here $N\ge2$, otherwise we can directly apply the base case $N=1$, see the end of this proof).
	By the proof of Lemma~\ref{lem:ObsInequ_L2} there exists an inner node $\nu_1$ of $G$ that is a boundary node of first order of $G$ and that has exactly one incident inner edge, cf. Fig. \ref{fig:network2_tree-shaped_reduction}.
	As in the proof of Lemma~\ref{lem:ObsInequ_L2}, we consider the subnetwork $G_1$ of $G$ that is constructed by removing all incident boundary edges of $\nu_1$.
	Then, $G_1$ is a tree-shaped network with $N-1$ inner nodes and by assumption there exists a constant $C_1>0$ such that
	\begin{align} \label{eq:ObsInequ_L2_G1_fric}
		\|(\delta_+, \delta_-)(t,\cdot)\|^2_{L^2(\E(G_1))}
		\le C_1 \sum_{\nu\in\V_\partial(G_1)}  \|(\delta_+, \delta_-)(\cdot, \nu)\|^2_{L^2([t-b_{N-1} \tfrac{\lmax}{c},\, t+b_{N-1} \tfrac{\lmax}{c}])}.
	\end{align}
	Our goal is to induce \eqref{eq:obsInequ_L2_fric} from \eqref{eq:ObsInequ_L2_G1_fric}.
	Since $\V_\partial(G_1)\subset\{\nu_1\}\cup \V_\partial(G)$, as a first step we have to estimate 
	$ \|(\delta_+^e, \delta_-^e)(\cdot, \nu_1)\|^2_{L^2([t-b_{N-1} \tfrac{\lmax}{c}, \, t+b_{N-1} \tfrac{\lmax}{c}])} $ by the right hand side of \eqref{eq:obsInequ_L2_fric}, where $e$ is the inner edge that is incident to $\nu_1$.
	Now, we consider the coupling conditions at the node $\nu_1$. Analogously to the proof of Lemma~\ref{lem:ObsInequ_L2} we can show 
	\begin{align*} 
		|\delta^{m+1}\out(s,\nu_1)|^2 + |\delta^{m+1}\In(s,\nu_1)|^2
		\le \left(\tfrac{3}{2} M^2 +8 (M-1)\right) \sum_{i=1}^{m} \left(|\delta^i\In(s,\nu_1)|^2 +|\delta^i\out(s,\nu_1)|^2 \right),
	\end{align*}
	where we denoted the edges as in Fig.~\ref{fig:network_coupling_condition} and $M$ denotes the the maximal number of edges that can be incident to a node in $G$.
	
	The main difference to the proof of Lemma~\ref{lem:ObsInequ_L2} is that the value of $\delta_\pm$ is no longer constant along the edges of the network.
	Let $\nu_{\partial_j}$ denote the boundary node that is incident to the edge~$j$, $j\in\{1,\ldots,m\}$, in $G$. Then,
	in order to estimate $\|(\delta_+^i, \delta_-^i)(\cdot, \nu_1)\|^2_{L^2([t-b_{N-1} \tfrac{\lmax}{c},\, t+b_{N-1} \tfrac{\lmax}{c}])}$,  
	we have to estimate the $L^2$-norm of $\delta^i_\pm(\cdot,\nu_1)$ by the $L^2$-norm of
	$\delta^i_\pm(\cdot,\nu_{\partial_i})$.
	For this, we use that Lemma~\ref{lem:reversed_obs_inequ} implies
	\begin{align*}
		&\int_{t-b_{N-1} \tfrac{\lmax}{c}}^{t+b_{N-1} \tfrac{\lmax}{c}}
		\left(|\delta^i_+(s,\nu_1)|^2 +|\delta^i_-(s,\nu_1)|^2 \right)ds\\
		&\le C \int_{t-b_{N-1} \tfrac{\lmax}{c}-2\tfrac{\ell^i}{c}}^{t+b_{N-1} \tfrac{\lmax}{c}+2\tfrac{\ell^i}{c}}
		\left(|\delta^i_+(s,\nu_{\partial_i})|^2 +|\delta^i_-(s,\nu_{\partial_i})|^2 \right)ds\\
		&\le C \int_{t-b_N\tfrac{\lmax}{c}}^{t+b_N\tfrac{\lmax}{c}}
		\left(|\delta^i_+(s,\nu_{\partial_i})|^2 +|\delta^i_-(s,\nu_{\partial_i})|^2 \right)ds
	\end{align*}
	for $i\in\{1,\ldots,m\}$.
	In summary, this implies
	\begin{align*}
		&\|(\delta_+^{m+1}, \delta_-^{m+1})(\cdot, \nu_1)\|^2_{L^2([t-b_{N-1} \tfrac{\lmax}{c}, \, t+b_{N-1} \tfrac{\lmax}{c}])}\\
		&\le \left(\tfrac{3}{2} M^2 +8 (M-1)\right) \int_{t-b_{N-1} \tfrac{\lmax}{c}}^{t+b_{N-1}\tfrac{\lmax}{c}} \sum_{i=1}^{m} \left(|\delta^i\In(s,\nu_1)|^2 +|\delta^i\out(s,\nu_1)|^2 \right)ds\\
		&\le C \left(\tfrac{3}{2} M^2 +8 (M-1)\right) \int_{t-b_{N} \tfrac{\lmax}{c}}^{t+b_{N}\tfrac{\lmax}{c}} \sum_{i=1}^{m} \left(|\delta^i\In(s,\nu_{\partial_i})|^2 +|\delta^i\out(s,\nu_{\partial_i})|^2 \right)ds\\
		&\le \tilde C \sum_{\nu\in\V_\partial(G)} \|(\delta_+, \delta_-)(\cdot, \nu)\|_{L^2([t-b_N\tfrac{\lmax}{c}, t+b_N\tfrac{\lmax}{c}])}^2.
	\end{align*}
	This concludes the first step of the proof of \eqref{eq:obsInequ_L2_fric} using \eqref{eq:ObsInequ_L2_G1_fric}.
	
	As a second step, we estimate $\|(\delta_+^j, \delta_-^j)(t,\cdot)\|^2_{L^2(0,\ell^j)}$ for all removed boundary edges\linebreak ${j=1,\ldots,m}$. Using the observability inequality for a single pipe, see Lemma~\ref{lem:obs_inequ_single_pipe}, we see
	\begin{align*}
		\int_0^{\ell^j} \left(|\delta^j_+(t,x)|^2+|\delta^j_-(t,x)|^2\right) dx
		&\le C \int_{t-\tfrac{\ell^j}{c}}^{t+\tfrac{\ell^j}{c}} \left(|\delta^j_+(s,\nu_{\partial_j})|^2+|\delta^j_-(s,\nu_{\partial_j})|^2\right) ds\\
		&\le  C \|(\delta_+^j, \delta_-^j)(\cdot,\nu_{\partial_j})\|^2_{L^2([t-\tfrac{\lmax}{c},\, t+\tfrac{\lmax}{c}])}
	\end{align*}
	for $j\in\{1,\ldots,m\}$.
	
	Together, we have shown that the induction hypothesis \eqref{eq:ObsInequ_L2_G1_fric} implies
	\begin{align*}
		\|(\delta_+, \delta_-)(t,\cdot)\|^2_{L^2(\E(G))}
		&\le \sum_{i=1}^{m} \|(\delta_+^i, \delta_-^i)(t,\cdot)\|^2_{L^2(0,\ell^j)}
		+\|(\delta_+, \delta_-)(t,\cdot)\|^2_{L^2(\E(G_1))}\\
		&\le C \sum_{\nu\in\V_\partial(G)}  \|(\delta_+, \delta_-)(\cdot, \nu)\|^2_{L^2([t-b_N \tfrac{\lmax}{c}, \, t+b_N \tfrac{\lmax}{c}])},
	\end{align*}
	i.e., we have shown that \eqref{eq:ObsInequ_L2_G1_fric} implies \eqref{eq:obsInequ_L2_fric}.
	
	It remains to show the base case of the induction. This means that we have to show the assertion for the network $G_{N-1}$, which is a star-shaped network, cf. Fig.~\ref{fig:star-shaped_network}.
	Denote the inner node by $\tilde \nu$, the edges by $e=1,\ldots, \tilde m$ and the boundary nodes by $\nu_{\partial_j}$, $j=1,\ldots, \tilde m$.
	Again, using the observability inequality for a single pipe, see Lemma~\ref{lem:obs_inequ_single_pipe}, we can show
	\begin{align*}
		&\|(\delta_+, \delta_-)(t,\cdot)\|^2_{L^2(\E(G_{N-1}))}
		=\sum_{i=1}^{\tilde m} \int_0^{\ell^i} \left(|\delta^i_+(t,x)|^2+|\delta^i_-(t,x)|^2\right) dx\\
		&\le  C \sum_{i=1}^{\tilde m} \|(\delta_+^i, \delta_-^i)(\cdot,\nu_{\partial_i})\|^2_{L^2([t-\tfrac{\lmax}{c},\, t+\tfrac{\lmax}{c}])}\\
		&\le  C \sum_{\nu\in\V_\partial(G_{N-1})}  \|(\delta_+, \delta_-)(\cdot, \nu)\|^2_{L^2([t- \tfrac{\lmax}{c}, \, t+\tfrac{\lmax}{c}])}.
	\end{align*}
	Together, this shows the assertion of the theorem.
\end{proof}

\begin{remark} \label{rem:limit_C3}
	Tracking the constants in the proof of Theorem~\ref{thm:ObsInequ_L2_fric} shows that, 
	for $\lambda\rightarrow 0$, the constant $C_3$ of the $L^2$-observability inequality for tree-shaped networks converges to a positive and bounded constant that depends on $c$, the number $N$ of inner nodes of the network and on the maximal number $M$ of edges that are incident to any node.
\end{remark}

Using the $L^2$-observability inequality stated in Theorem \ref{thm:ObsInequ_L2_fric}, we can show that the solution $\delta_\pm$ of the difference system \eqref{eq:system_diff_fric}--\eqref{eq:diff_CC_fric} converges exponentially to zero in the long time limit for tree-shaped networks,
i.e., the state of the observer system converges exponentially towards the original solution.

\begin{thm}  \label{thm:expSynchr_L2_fric}
	Let $G=(\V, \E)$ be a tree-shaped network and let $\delta_\pm$ be a solution of the system \eqref{eq:system_diff_fric}--\eqref{eq:diff_CC_fric}. Then there exist constants $\mu>0$, $C_1>0$ such that 
	\begin{align*}
		\|(\delta_+, \delta_-)(t,\cdot)\|^2_{L^2(\E)}
		\le C_1 \exp(-\mu t) \quad \forall t>0.
	\end{align*}
\end{thm}

\begin{proof}
	Let some $t>0$ be given. As in the proof of Theorem~\ref{thm:expSynchr_L2} and \cite[Theorem~4]{Gugat2021}, we multiply \eqref{eq:system_diff_fric} by $\delta_\pm^e$ and integrate over $[t-\tilde t, t+\tilde t]\times[0,\ell^e]$ for some $\tilde t\in(0,t)$.
	This yields
	\begin{align*}
		\int_{t-\tilde t}^{t+\tilde t} \int_{0}^{\ell^e} \delta_\pm^e(s,x) \dt \delta_\pm^e(s,x)\, dx\, ds
		=&\,(\mp c) \int_{t-\tilde t}^{t+\tilde t} \int_{0}^{\ell^e} \delta_\pm^e(s,x) \dx \delta_\pm^e(s,x)\, dx\, ds\\
		&\mp \lambda \int_{t-\tilde t}^{t+\tilde t} \int_{0}^{\ell^e} \delta_\pm^e(s,x) \left(\delta_+^e(s,x)-\delta_-^e(s,x)  \right)\, dx\, ds ,
	\end{align*}
	i.e.,
	\begin{align*}
		\tfrac{1}{2}  \int_{0}^{\ell^e} \left[  | \delta_\pm^e(s,x)|^2 \right]_{s=t-\tilde t}^{t+\tilde t} \, dx
		=&\, (\mp \tfrac{c}{2})  \int_{t-\tilde t}^{t+\tilde t} \left[  | \delta_\pm^e(s,x)|^2 \right]_{x=0}^{\ell^e} \, ds\\
		& \mp \lambda \int_{t-\tilde t}^{t+\tilde t} \int_{0}^{\ell^e} \delta_\pm^e(s,x) \left(\delta_+^e(s,x)-\delta_-^e(s,x)  \right)\, dx\, ds.
	\end{align*}
	Summing up over all pipes $e\in\E$, we get the following inequality for the $L^2$-norm
	\begin{align*}
		&\|(\delta_+, \delta_-)(t+\tilde t,\cdot)\|^2_{L^2(\E)} - \|(\delta_+, \delta_-)(t-\tilde t,\cdot)\|^2_{L^2(\E)}\\
		&= (-c) \sum_{e\in\E} \int_{t-\tilde t}^{t+\tilde t} \left[  | \delta_+^e(s,x)|^2  -| \delta_-^e(s,x)|^2 \right]_{x=0}^{\ell^e} \, ds
		-2\lambda  \int_{t-\tilde t}^{t+\tilde t} \int_{0}^{\ell^e}  \left|\delta_+^e(s,x)-\delta_-^e(s,x)  \right|^2\, dx\, ds \\
		&\le c  \int_{t-\tilde t}^{t+\tilde t} \sum_{\nu\in\V} \sum_{e\in\E(\nu)}   \big( |\delta\out^e(s,\nu)|^2  -| \delta\In^e(s,\nu)|^2 \big)\, ds.
	\end{align*}
	The rest of the proof is analogously to the proof of Theorem~\ref{thm:expSynchr_L2}.
\end{proof}

\begin{remark}
		Theorem ~\ref{thm:expSynchr_L2_fric} shows that $\|(\delta_+, \delta_-)(t,\cdot)\|^2_{L^2(\E)}$ converges exponentially to zero with parameter $\mu :=-\tfrac{\ln(\tilde C)}{2T}$, where $\tilde C :=\tfrac{1}{1+\tfrac{c}{C_3}}$. Using Remark~\ref{rem:limit_C3} we can see that $\mu$ converges to a small, but positive constant $\mu_0>0$ for $\lambda\rightarrow 0$,	
		i.e., the convergence is uniform in the $\lambda\rightarrow 0$ limit.
		This means that the exponential synchronization is guaranteed by the boundary condition~\eqref{eq:diff_BC_fric} that accounts for the boundary measurements and does not build on a possible decay due to the friction term.	
\end{remark}

\begin{remark}
	We can extend the exponential convergence stated in Theorem~\ref{thm:expSynchr_L2_fric} to a semilinear system with more general friction law, 
	i.e., to the system
	\begin{align}
		&\begin{pmatrix}
			\dt R_+^e \\ \dt R_-^e
		\end{pmatrix}
		+\begin{pmatrix}
			c & 0\\ 0 &-c
		\end{pmatrix}
		\begin{pmatrix}
			\dx R_+^e \\ \dx R_-^e
		\end{pmatrix}  
		= \begin{pmatrix}
			-\sigma( R_+^e,R_-^e)\\ \sigma( R_+^e,R_-^e)
		\end{pmatrix}, 
		&& e\in\E, \label{eq:system_fric_semilin}
	\end{align}
	with $\sigma( R_+,R_-)=\tilde{\sigma}(R_+-R_-)$, where $\tilde{\sigma}(0)=0$, $\tilde{\sigma}$ is Lipschitz-continuous with Lipschitz-constant $L_\sigma$ and $\tilde{\sigma}'\ge 0$
	on  the convex hull of the union of the images of $(R_+-R_-)$ and $(S_+-S_-)$.
	For
	\begin{align*}
		\sigma( R_+,R_-)=\gamma |R_+-R_-| (R_+-R_-)
	\end{align*}
	with friction parameter $\gamma\ge0$,
	an existence result for \eqref{eq:system_fric_semilin} is provided by \cite[Theorem 1]{Gugat2021} and an $L^2$-observability inequality for a single pipe is given in \cite[Theorem 2]{Gugat2021}.

	The difference system \eqref{eq:system_diff_fric} corresponding to \eqref{eq:system_fric_semilin} is given by
	\begin{align}
		&\begin{pmatrix}
			\dt \delta_+^e \\ \dt \delta_-^e
		\end{pmatrix}
		+\begin{pmatrix}
			c & 0\\ 0 &-c
		\end{pmatrix}
		\begin{pmatrix}
			\dx \delta_+^e \\ \dx \delta_-^e
		\end{pmatrix}  
		= \begin{pmatrix}
			-\left(\sigma( R_+^e,R_-^e)-\sigma( S_+^e,S_-^e)\right)\\ \sigma( R_+^e,R_-^e)-\sigma( S_+^e,S_-^e)
		\end{pmatrix}, 
		&& e\in\E. \label{eq:system_diff_fric_semilin}
	\end{align}
	The main step to extend the exponential convergence to the semilinear system is to replace the estimation at the beginning of the proof of Lemma~\ref{lem:reversed_obs_inequ} by 
	\begin{align*}
		&|\delta^e_+|^2(s,\ell^e)-|\delta^e_+|^2(t-\tfrac{1}{2} 	\tfrac{\ell^e}{c},\tfrac{1}{2}\ell^e-c(s-t))
		=\int_{t-\tfrac{1}{2} \tfrac{\ell^e}{c}}^{s} \tfrac{d}{d \tilde s}
		\left(|\delta^e_+|^2(\tilde s, \ell^e+c(\tilde s - s)) \right)  d\tilde s\\
		&= \int_{t-\tfrac{1}{2} \tfrac{\ell^e}{c}}^{s}
		(-2) \delta_+^e  \left(\sigma( R_+^e,R_-^e)-\sigma( S_+^e,S_-^e)\right)(\tilde s, \ell^e+c(\tilde s - s)) d\tilde s\\
		&\le L_\sigma
		\int_{t-\tfrac{1}{2} \tfrac{\ell^e}{c}}^{s}
		(|\delta_+^e|^2+|\delta_-^e|^2)(\tilde s, \ell^e+c(\tilde s - s)) d\tilde s,
	\end{align*}
	where we 
	have used the mean value theorem for $\tilde\sigma(\cdot)$ as well as Young's inequality in the last step. 
	The estimation of $\G_+$ can be handled similarly. Therefore the proof of Lemma~\ref{lem:reversed_obs_inequ} can be extended to the semilinear case.
	Note that the proof of Theorem~\ref{thm:ObsInequ_L2_fric} does not depend on the explicit form of the friction term.
	The estimation at the beginning of the proof of Theorem~\ref{thm:expSynchr_L2_fric} can again be adapted using the mean value theorem for $\tilde \sigma(\cdot)$ and the rest of the proof of Theorem~\ref{thm:expSynchr_L2_fric} is independent of the explicit form of the friction term.
	Therefore the exponential convergence can be extended to the semilinear system.
\end{remark}

\section*{Acknowledgement}
The authors are grateful for financial support by the German Science Foundation (DFG) via grant TRR~154 (\emph{Mathematical modelling, simulation and optimization using the example of gas networks}), sub-project C05 (Project
239904186). The work of J.G. is also supported by the Graduate School CE within Computational Engineering at Technische Universität Darmstadt. 


\end{document}